\documentclass[11pt,a4paper]{article}
\usepackage{amsfonts}
\usepackage{amsmath}
\usepackage{amssymb}
\usepackage{amstext}
\usepackage{amsthm}
\usepackage[toc,page]{appendix}
\usepackage{authblk}
\usepackage{cancel}
\usepackage{cite}
\usepackage{datetime}
\usepackage{enumerate}
\usepackage[cm]{fullpage}
\usepackage{graphicx}
\usepackage[latin1]{inputenc}
\usepackage{multirow}
\usepackage{pgfplots}
\usepackage{subcaption}
\usepackage{tikz}
\usepackage{times}
\usepackage[normalem]{ulem}
\usepackage{setspace}
\usepackage{units}
\usepackage{url}
\usepackage{xcolor}
\usepackage{hyperref}

\newcommand{\ajcomment}[1]{}
\renewcommand{\ajcomment}[1]{\textcolor{red}{\bf $\star$}\marginpar{\color{red} {\it \scriptsize \raggedright #1}}}

\usepackage[font=small]{caption}

\usetikzlibrary{shapes,arrows,plotmarks,matrix,positioning,fit,calc,3d}
\newcommand{\iid}{\stackrel{\mathrm{iid}}{\sim}}

\newtheorem{theorem}{Theorem}[section]

\newtheorem{proposition}[theorem]{Proposition}

\newtheorem{remark}[theorem]{Remark}
\newtheorem{mydef}{Definition}[section]
\newtheorem{assumptions}{Assumptions}

\DeclareMathOperator*{\argmin}{argmin}

\newlength\figureheight
\newlength\figurewidth

\begin{document}

\title{Convergence of the $k$-Means Minimization\\Problem using $\Gamma$-Convergence}
\author[1]{Matthew Thorpe}
\author[1]{Florian Theil}
\author[1]{Adam M. Johansen}
\author[2]{Neil Cade}
\affil[1]{University of Warwick, Coventry, CV4 7AL, United Kingdom }
\affil[2]{Selex-ES, Luton, LU1 3PG, United Kingdom}
%\date{February 2014}
%\date{\today \; at \currenttime}
\date{}

\maketitle
\abstract{
The $k$-means method is an iterative clustering algorithm which associates each observation with one of $k$ clusters.
It traditionally employs cluster centers in the same space as the observed data.
By relaxing this requirement, it is possible to apply the $k$-means method to infinite dimensional problems, for example multiple target tracking and smoothing problems in the presence of unknown data association.
Via a $\Gamma$-convergence argument, the associated optimization problem is shown to converge in the sense that both the $k$-means minimum and minimizers converge in the large data limit to quantities which depend upon the observed data only through its distribution.
The theory is supplemented with two examples to demonstrate the range of problems now accessible by the $k$-means method.
The first example combines a non-parametric smoothing problem with unknown data association.
The second addresses tracking using sparse data from a network of passive sensors.
}

\section{Introduction \label{sec:intro}}

The $k$-means algorithm \cite{lloyd82} is a technique for assigning each of a collection of observed data to exactly one of $k$ clusters, each of which has a unique center, in such a way that each observation is assigned to the cluster whose center is closest to that observation in an appropriate sense.

The $k$-means method has traditionally been used with limited scope.
Its usual application has been in Euclidean spaces which restricts its application to finite dimensional problems.
There are relatively few theoretical results using the $k$-means methodology in infinite dimensions of which \cite{biau08,canas12,cuesta07, laloe10, lember03, linder02, tarpey03} are the only papers known to the authors.
In the right framework, post-hoc track estimation in multiple target scenarios with unknown data association can be viewed as a clustering problem and therefore accessible to the $k$-means method.
In such problems one typically has finite-dimensional data, but would wish to estimate infinite dimensional tracks with the added complication of unresolved data association.
It is our aim to propose and characterize a framework for the $k$-means method which can deal with this problem.

A natural question to ask of any clustering technique is whether the estimated clustering stabilizes as more data becomes available.
More precisely, we ask whether certain estimates converge, in an appropriate sense, in the large data limit.
In order to answer this question in our particular context we first establish a related optimization problem and make precise the notion of convergence.

Consistency of estimators for ill-posed inverse problems has been well studied, for example \cite{dashti13, osullivan86}, but without the data association problem.
In contrast to standard statistical consistency results, we do not assume that there exists a structural relationship between the optimization problem and the data-generating process in order to establish convergence to true parameter values in the large data limit; rather, we demonstrate convergence to the solution of a related limiting problem.

This paper shows the convergence of the minimization problem associated with the $k$-means method in a framework that is general enough to include examples where the cluster centers are not necessarily in the same space as the data points.
In particular we are motivated by the application to infinite dimensional problems, e.g. the smoothing-data association problem.
The smoothing-data association problem is the problem of associating data points $\{(t_i,z_i)\}_{i=1}^n \subset [0,1]\times\mathbb{R}^\kappa$ to unknown trajectories $\mu_j:[0,1]\to\mathbb{R}^\kappa$ for $j=1,2,\dots,k$.
By treating the trajectories $\mu_j$ as the cluster centers one may approach this problem using the $k$-means methodology.
The comparison of data points to cluster centers is a pointwise distance: $d((t_i,z_i),\mu_j)=|\mu_j(t_i)-z_i|^2$ (where $|\cdot|$ is the Euclidean norm on $\mathbb{R}^\kappa$).
To ensure the problem is well-posed some regularization is also necessary.
For $k=1$ the problem reduces to smoothing and coincides with the limiting problem studied in \cite{hall05}.
We will discuss the smoothing-data association problem more in Section~\ref{sec:MTT:app}.

Let us now introduce the notation for our variational approach.
The $k$-means method is a strategy for partitioning a data set $\Psi_n = \{\xi_i\}_{i=1}^n \subset X$ into $k$ clusters where each cluster has center $\mu_j$ for $j=1,2,\dots, k$.
First let us consider the special case when $\mu_j\in X$.
The data partition is defined by associating each data point with the cluster center closest to it which is measured by a cost function $d:X\times X \to [0,\infty)$.
Traditionally the $k$-means method considers Euclidean spaces $X=\mathbb{R}^\kappa$, where typically we choose $d(x,y)=|x-y|^2=\sum_{i=1}^\kappa(x_i-y_i)^2$.
We define the energy for a choice of cluster centers given data by
\begin{align*}
f_n: X^k & \to \mathbb{R} &
f_n(\mu|\Psi_n) & = \frac{1}{n} \sum_{i=1}^n \bigwedge_{j=1}^k d(\xi_i,\mu_j),
\end{align*}
where for any $k$ variables, $a_1, a_2,\dots, a_k$,
\( \bigwedge_{j=1}^k a_j := \min\{a_1,\ldots,a_k\}. \)
The optimal choice of $\mu$ is that which minimizes $f_n(\cdot|\Psi_n)$.
We define
\begin{equation*}
\hat{\theta}_n = \min_{\mu\in X^k} f_n(\mu|\Psi_n) \in \mathbb{R}.
\end{equation*}

An associated ``limiting problem'' can be defined
\begin{equation*}
\theta = \min_{\mu\in X^k} f_\infty(\mu)
\end{equation*}
where we assume, in a sense which will be made precise later, that $\xi_i\iid P$ for some suitable probability distribution, $P$, and define
\[ f_\infty(\mu) = \int \bigwedge_{j=1}^k d(x,\mu_j) P(\text{d} x). \]
In Section \ref{sec:cons} we validate the formulation by first showing that, under regularity conditions and with probability one, the minimum energy converges: $\hat{\theta}_n \to \theta$.
And secondly by showing that (up to a subsequence) the minimizers converge: $\mu^n \to \mu^\infty$ where $\mu^n$ minimizes $f_n$ and  $\mu^\infty$ minimizes $f_\infty$ (again with probability one).

In a more sophisticated version of the $k$-means method the requirement that $\mu_j \in X$ can be relaxed.
We instead allow $\mu=(\mu_1,\mu_2,\dots,\mu_k)\in Y^k$ for some other Banach
space, $Y$, and define $d$ appropriately.
This leads to interesting statistical questions.
When $Y$ is infinite dimensional even establishing whether or not a minimizer exists is non-trivial.

When the cluster center is in a different space to the data, bounding the set of minimizers becomes less natural.
For example, consider the smoothing problem in which one wishes to fit a continuous function to a set of data points.
The natural choice of cost function is a pointwise distance of the data to the curve.
The optimal solution is for the cluster center to interpolate the data points: in the limit the cluster center may no longer be well defined.
In particular we cannot hope to have converging sequences of minimizers.

In the smoothing literature this problem is prevented by using a regularization term $r:Y^k\to \mathbb{R}$.
For a cost function $d:X\times Y\to [0,\infty)$ the energies $f_n(\cdot|\Psi_n),f_\infty(\cdot):Y^k\to \mathbb{R}$ are redefined
\begin{align*}
f_n(\mu|\Psi_n) & = \frac{1}{n} \sum_{i=1}^n \bigwedge_{j=1}^k d(\xi_i,\mu_j) + \lambda_n r(\mu) \\
f_\infty(\mu) & = \int \bigwedge_{j=1}^k d(x,\mu_j) P(\text{d} x) + \lambda r(\mu).
\end{align*}
Adding regularization changes the nature of the problem so we commit time in Section \ref{sec:MTT} to justifying our approach.
Particularly we motivate treating $\lambda_n = \lambda$ as a constant independent of $n$.
We are able to repeat the analysis from Section 4; that is to establish that the minimum and a subsequence of minimizers still converge.

Early results assumed $Y=X$ were Euclidean spaces and showed the convergence of minimizers to the appropriate limit \cite{hartigan78,pollard81}.
The motivation for the early work in this area was to show consistency of the methodology.
In particular this requires there to be an underlying `truth'.
This requires the assumption that there exists a unique minimizer to the limiting energy.
These results do not hold when the limiting energy has more than one minimizer \cite{ben-david07}.
In this paper we discuss only the convergence of the method and as such require no assumption as to the existence or uniqueness of a minimizer to the limiting problem.
Consistency has been strengthened to a central limit theorem in \cite{pollard82} also assuming a unique minimizer to the limiting energy.
Other rates of convergence have been shown in \cite{antos05, bartlett98, chou94, linder94}.
In Hilbert spaces there exist convergence results and rates of convergence for the minimum.
In~\cite{biau08} the authors show that $|f_n(\mu^n)-f_\infty(\mu^\infty)|$ is of order $\frac{1}{\sqrt{n}}$, however, there are no results for the convergence of minimizers.
Results exist for $k\to \infty$, see for example~\cite{canas12} (which are also valid for $Y\neq X$).

Assuming that $Y=X$, the convergence of the minimization problem in a reflexive and separable Banach space has been proved in \cite{linder02} and a similar result in metric spaces in \cite{lember03}.
In \cite{laloe10}, the existence of a weakly converging subsequence was inferred using the results of \cite{linder02}.

In the following section we introduce the notation and preliminary material used in this paper.

We then, in Section~\ref{sec:cons}, consider convergence in the special case when the cluster centers are in the same space as the data points, i.e. $Y=X$.
In this case we don't have an issue with well-posedness as the data has the same dimension as the cluster centers.
For this reason we use energies defined without regularization.
Theorem~\ref{thm:cons} shows that the minimum converges, i.e. $\hat{\theta}_n\to \theta$ as $n\to \infty$, for almost every sequence of observations and furthermore we have a subsequence $\mu^{n_m}$ of minimizers of $f_{n_m}$ which weakly converge to some $\mu^\infty$ which minimizes $f_\infty$.

This result is generalized in Section~\ref{sec:MTT} to an arbitrary $X$ and $Y$.
The analogous result to Theorem~\ref{thm:cons} is Theorem~\ref{thm:kmeanscons}.
We first motivate the problem and in particular our choice of scaling in the regularization in Section \ref{sec:MTT:reg} before proceeding to the results in Section \ref{sec:MTT:theory}.
Verifying the conditions on the cost function $d$ and regularization term $r$ is non-trivial and so we show an application to the smoothing-data association problem in Section \ref{sec:MTT:app}.

To demonstrate the generality of the results in this paper, two applications are considered in Section~\ref{sec:examples}.
The first is the data association and smoothing problem. We show the minimum converging as the data size increases.
We also numerically investigate the use of the $k$-means energy to determine whether two targets have crossed tracks.
The second example uses measured times of arrival and amplitudes of signals
from moving sources that are received across a network of three sensors. The
cluster centers are the source trajectories in $\mathbb{R}^2$.

\section{Preliminaries \label{sec:not}}

In this section we introduce some notation and background theory which will be used in Sections~\ref{sec:cons} and \ref{sec:MTT} to establish our convergence results.
In these sections we show the existence of optimal cluster centers using the direct method.
By imposing conditions, such that our energies are weakly lower semi-continuous, we can deduce the existence of minimizers.
Further conditions ensure the minimizers are uniformly bounded.
The $\Gamma$-convergence framework (e.g. \cite{braides02,dalmaso93}) allows us to establish the convergence of the minimum and also the convergence of minimizers.

We have the following definition of $\Gamma$-convergence with respect to weak convergence.

\begin{mydef}[$\Gamma$-convergence]
\label{def:gamcon}
A sequence $f_n :A\to \mathbb{R}\cup \{\pm\infty\}$ on a Banach space $(A,\|\cdot\|_A)$ is said to \textit{$\Gamma$-converge} on the domain $A$ to $f_\infty :A\to \mathbb{R}\cup \{\pm\infty\}$ with respect to weak convergence on $A$, and we write $f_\infty = \Gamma\text{-}\lim_n f_n$, if for all $x\in A$ we have
\begin{itemize}
\item[(i)] (liminf inequality) for every sequence $(x_n)$ weakly converging to $x$
\[ f_\infty(x) \leq \liminf_n f_n(x_n); \]
\item[(ii)] (recovery sequence) there exists a sequence $(x_n)$ weakly converging to $x$ such that
\[ f_\infty(x) \geq \limsup_n f_n(x_n). \]
\end{itemize}
\end{mydef}

When it exists the $\Gamma$-limit is always weakly lower semi-continuous, and thus admits minimizers.
An important property of $\Gamma$-convergence is that it implies the convergence of minimizers.
In particular, we will make extensive use of the following well-known result.

\begin{theorem}[Convergence of Minimizers]
\label{thm:conmin}
Let $f_n: A\to \mathbb{R}$ be a sequence of functionals on a Banach space $(A,\|\cdot\|_A)$ and assume that there exists $N>0$ and a weakly compact subset $K\subset A$ with
\[ \inf_A f_n = \inf_K f_n \quad \forall n>N. \]
If $f_\infty = \Gamma\text{-}\lim_n f_n$ and $f_\infty$ is not identically $\pm\infty$ then
\[ \min_A f_\infty = \lim_n \inf_A f_n. \]
Furthermore if each $f_n$ is weakly lower semi-continuous then for each $f_n$ there exists a minimizer $x_n\in K$ and any weak limit point of $x_n$ minimizes $f_\infty$.
Since $K$ is weakly compact there exists at least one weak limit point.
\end{theorem}
A proof of the theorem can be found in \cite[Theorem 1.21]{braides02}.

The problems which we address involve random observations.
We assume throughout the existence of a probability space
$(\Omega,\mathcal{F},\mathbb{P})$, rich enough to support a countably infinite
sequence of such observations, $\xi_1^{(\omega)},\ldots$.
All random elements are defined upon this common probability space and all stochastic quantifiers are to be understood as acting with respect to $\mathbb{P}$ unless otherwise stated.
Where appropriate, to emphasize the randomness of the functionals $f_n$, we will write $f^{(\omega)}_n$ to indicate the functional associated with the particular observation sequence $\xi_1^{(\omega)},\ldots,\xi_n^{(\omega)}$ and we allow $P_n^{(\omega)}$ to denote the associated empirical measure.

We define the support of a (probability) measure to be the smallest closed set such that the complement is null.

For clarity we often write integrals using operator notation.
I.e. for a measure $P$, which is usually a probability distribution, we write
\[ Ph = \int h(x) \; P(\text{d}x). \]
For a sequence of probability distributions, $P_n$, we say that $P_n$ converges weakly to $P$ if
\[ P_nh \to Ph \quad \quad \text{for all bounded and continuous } h \]
and we write $P_n\Rightarrow P$.
With a slight abuse of notation we will sometimes write $P(U):=P\mathbb{I}_U$ for a measurable set $U$.

For a Banach space $A$ one can define the dual space $A^*$ to be the space of all bounded and linear maps over $A$ into $\mathbb{R}$ equipped with the norm $\|F\|_{A^*} = \sup_{x\in A} |F(x)|$.
Similarly one can define the second dual $A^{**}$ as the space of all bounded and linear maps over $A^*$ into $\mathbb{R}$.
Reflexive spaces are defined to be spaces $A$ such that $A$ is isometrically isomorphic to $A^{**}$.
These have the useful property that closed and bounded sets are weakly compact.
For example any $L^p$ space (with $1<p<\infty$) is reflexive, as is any Hilbert space (by the Riesz Representation Theorem: if $A$ is a Hilbert space then $A^*$ is isometrically isomorphic to $A$).

A sequence $x_n\in A$ is said to weakly convergence to $x\in A$ if $F(x_n)\to F(x)$ for all $F\in A^*$.
We write $x_n\rightharpoonup x$.
We say a functional $G:A\to \mathbb{R}$ is weakly continuous if $G(x_n)\to G(x)$ whenever $x_n \rightharpoonup x$ and strongly continuous if $G(x_n)\to G(x)$ whenever $\|x_n-x\|_A\to 0$.
Note that weak continuity implies strong continuity.
Similarly a functional $G$ is weakly lower semi-continuous if $\liminf_{n\to \infty} G(x_n)\geq G(x)$ whenever $x_n\rightharpoonup x$.

We define the Sobolev spaces $W^{s,p}(I)$ on $I\subseteq \mathbb{R}$ by
\[ W^{s,p} = W^{s,p}(I) = \left\{ f:I\to \mathbb{R} \text{ s.t. } \partial^i f \in L^p(I) \text{ for } i=0,\dots, s \right\} \]
where we use $\partial$ for the weak derivative, i.e. $g=\partial f$ if for all $\phi\in C_c^\infty(I)$ (the space of smooth functions with compact support)
\[ \int_I f(x) \frac{\mathrm{d}\phi}{\mathrm{d}x}(x) \; \text{d} x = - \int_I g(x) \phi(x) \; \text{d} x. \]
In particular, we will use the special case when $p=2$ and we write $H^s=W^{s,2}$.
This is a Hilbert space with norm:
\[ \|f\|_{H^s}^2 = \sum_{i=0}^s \| \partial^i f\|_{L^2}^2. \]

For two real-valued and positive sequences $a_n$ and $b_n$ we write $a_n \lesssim b_n$ if $\frac{a_n}{b_n}$ is bounded.
For a space $A$ and a set $K\subset A$ we write $K^c$ for the complement of $K$ in $A$, i.e. $K^c=A\setminus K$.

\section{Convergence when \texorpdfstring{$Y=X$}{Y=X} \label{sec:cons}}

We assume we are given data points $\xi_i\in X$ for $i=1,2,\dots$ where $X$ is a reflexive and separable Banach space with norm $\|\cdot\|_X$ and Borel $\sigma$-algebra $\mathcal{X}$.
These data points realize a sequence of $\mathcal{X}$-measurable random elements on $(\Omega, \mathcal{F}, \mathbb{P})$ which will also be denoted, with a slight abuse of notation, $\xi_i$.

We define
\begin{align}
f_n^{(\omega)}:X^k & \to \mathbb{R}, \quad  f_n^{(\omega)}(\mu) = P_n^{(\omega)} g_{\mu} = \frac{1}{n} \sum_{i=1}^n \bigwedge_{j=1}^k d(\xi_i^{(\omega)},\mu_j) \label{eq:fn} \\
f_\infty:X^k & \to \mathbb{R}, \quad  f_\infty (\mu) = P g_{\mu} = \int_X \bigwedge_{j=1}^k d(x,\mu_j) P(\text{d} x) \label{eq:finfty}
\end{align}
where
\[ g_{\mu}(x) = \bigwedge_{j=1}^k d(x,\mu_j), \]
$P$ is a probability measure on $(X,\mathcal{X})$, and empirical measure
$P_n^{(\omega)}$ associated with $\xi_1^{(\omega)},\ldots,\xi_n^{(\omega)}$ is defined by
\[ P_n^{(\omega)} h = \frac{1}{n} \sum_{i=1}^n h(\xi_i^{(\omega)}) \]
for any $\mathcal{X}$-measurable function $h: X\to \mathbb{R}$.
We assume $\xi_i$ are iid according to $P$ with $P=\mathbb{P}\circ \xi^{-1}_i$.

We wish to show
\begin{equation} \label{eq:limeq}
\hat{\theta}_n^{(\omega)} \to \theta \quad \text{for almost every } \omega \text{ as } n\to \infty
\end{equation}
where
\begin{align*}
\hat{\theta}_n^{(\omega)} & = \inf_{\mu\in X^k} f_n^{(\omega)}(\mu) \\
\theta & = \inf_{\mu\in X^k} f_\infty(\mu).
\end{align*}
We define $\|\cdot\|_k :X^k\to [0,\infty)$ by
\begin{equation} \label{eq:cons:norm}
\| \mu\|_k := \max_j \|\mu_j\|_X \quad \text{for } \mu=(\mu_1, \mu_2,\dots, \mu_k) \in X^k.
\end{equation}
The reflexivity of $(X,\|\cdot\|_X)$ carries through to $(X^k,\|\cdot\|_k)$.

Our strategy is similar to that of~\cite{pollard81} but we embed the methodology into the $\Gamma$-convergence framework.
We show that \eqref{eq:finfty} is the $\Gamma$-limit in Theorem~\ref{thm:gamcon} and that minimizers are bounded in Proposition~\ref{lem:bdd}.
We may then apply Theorem~\ref{thm:conmin} to infer \eqref{eq:limeq} and the existence of a weakly converging subsequence of minimizers.

The key assumptions on $d$ and $P$ are given in Assumptions~\ref{ass:d}.
The first assumption can be understood as a `closeness' condition for the space $X$ with respect to $d$.
If we let $d(x,y)=1$ for $x\neq y$ and $d(x,x)=0$ then our cost function $d$ does not carry any information on how far apart two points are.
Assume there exists a probability density for $P$ which has unbounded support.
Then $f_n^{(\omega)}(\mu)\geq \frac{n-k}{n}$ (for almost every $\omega$), with equality when we choose $\mu_j\in \{\xi_i^{(\omega)}\}_{i=1}^n$.
I.e. any set of $k$ unique data points will minimize $f_n^{(\omega)}$.
Since our data points are unbounded we may find a sequence $\|\xi_{i_n}^{(\omega)}\|_X\to \infty$.
Now we choose $\mu_1^n=\xi_{i_n}^{(\omega)}$ and clearly our cluster center is unbounded.
We see that this choice of $d$ violates the first assumption.
We also add a moment condition to the upper bound to ensure integrability.
Note that this also implies that $Pd(\cdot,0)\leq \int_X M(\|x\|) \; P(\text{d}x)<\infty$ so
$f_\infty(0)<\infty$ and, in particular, that $f_\infty$ is not identically infinity.

The second assumption is slightly stronger condition on $d$ than a weak lower semi-continuity condition in the first variable and strong continuity in the second variable.
The condition allows the application of Fatou's lemma for weakly converging probabilities, see~\cite{feinberg14}.

The third assumption allows us to view $d(\xi_i,y)$ as a collection of random variables.
The fourth implies that we have at least $k$ open balls with positive probability and therefore we are not overfitting clusters to data.

\begin{assumptions} \label{ass:d}
We have the following assumptions on $d:X\times X\to [0,\infty)$ and $P$.
\begin{enumerate}
\item[1.1.] \label{ass:d:norm} There exist continuous, strictly increasing functions $m,M:[0,\infty)\to [0,\infty)$ such that
\[ m(\|x-y\|_X) \leq d(x,y) \leq M(\|x-y\|_X) \quad \text{for all } x,y\in X \]
with $\lim_{r\to \infty}m(r)= \infty$, $M(0)=0$, there exists $\gamma<\infty$ such that $M(\|x+y\|_X)\leq \gamma M(\|x\|_X) + \gamma M(\|y\|_X)$ and finally $\int_X M(\|x\|_X) \; P(\text{d}x) < \infty$ (and $M$ is measurable).
\item[1.2.] \label{ass:d:cts} For each $x,y\in X$ we have that if $x_m\to x$ and $y_n\rightharpoonup y$ as $n,m\to \infty$ then
\[ \liminf_{n,m\to \infty} d(x_m,y_n) \geq d(x,y) \quad \text{and} \quad \lim_{m\to \infty} d(x_m,y) = d(x,y). \]
\item[1.3.] \label{ass:d:meas} For each $y\in X$ we have that $d(\cdot,y)$ is $\mathcal{X}$-measurable.
\item[1.4.] \label{ass:d:fit} There exist $k$ different centers $\mu^\dagger_j\in X$, $j=1,2,\dots,k$ such that for all $\delta>0$
\[ P(B(\mu_j^\dagger,\delta)) > 0 \quad \quad \quad \forall \; j=1,2,\dots,k \]
where $B(\mu,\delta):=\{x\in X: \|\mu-x\|_X< \delta \}$.
\end{enumerate}
\end{assumptions}

We now show that for a particular common choice of cost function, $d$, Assumptions~\ref{ass:d}.1 to~\ref{ass:d}.3 hold.

\begin{remark} \label{rem:dcond}
For any $p>0$ let $d(x,y)=\|x-y\|_X^p$ then $d$ satisfies Assumptions~\ref{ass:d}.1 to~\ref{ass:d}.3.
\end{remark}

\begin{proof}
Taking $m(r)=M(r)=r^p$ we can bound $m(\|x-y\|_X)\leq d(x,y) \leq M(\|x-y\|_X)$ and $m,M$ clearly satisfy $m(r)\to \infty$, $M(0)=0$, are strictly increasing and continuous.
One can also show that
\[ M(\|x+y\|_X) \leq 2^{p-1} \left( \|x\|^p_X + \|y\|^p_X \right) \]
hence Assumption~\ref{ass:d}.1 is satisfied.

Let $x_m\to x$ and $y_n\rightharpoonup y$.
Then
\begin{align*}
\liminf_{n,m\to\infty} d(x_m,y_n)^{\frac{1}{p}} & = \liminf_{n,m\to \infty} \|x_m-y_m\|_X \\
 & \geq \liminf_{n,m\to \infty} \left( \|y_n-x\|_X -\|x_m-x\|_X \right) \\
 & = \liminf_{n\to \infty} \|y_n-x\|_X \quad \text{since } x_m\to x \\
 & \geq \|y-x\|_X
\end{align*}
where the last inequality follows as a consequence of the Hahn-Banach Theorem and the fact that $y_n-x\rightharpoonup y-x$ which implies $\liminf_{n\to \infty} \|y_n-x\|_X\geq \|y-x\|_X$.
Clearly $d(x_m,y)\to d(x,y)$ and so Assumption~\ref{ass:d}.2 holds.

The third assumption holds by the Borel measurability of metrics on complete separable metric spaces.
\end{proof}

We now state the first result of the paper which formalizes the understanding that $f_\infty$ is the limit of $f_n^{(\omega)}$.

\begin{theorem}
\label{thm:gamcon}
Let $(X,\|\cdot\|_X)$ be a reflexive and separable Banach space with Borel $\sigma$-algebra, $\mathcal{X}$; let $\{\xi_i\}_{i\in\mathbb{N}}$ be a sequence of independent $X$-valued random elements with common law $P$.
Assume $d:X\times X\to [0,\infty)$ and that $P$ satisfies the conditions in Assumptions~\ref{ass:d}.
Define $f^{(\omega)}_n:X^k \to\mathbb{R}$ and $f_\infty:X^k\to\mathbb{R}$ by \eqref{eq:fn} and \eqref{eq:finfty} respectively.
Then
\[ f_\infty = \Gamma\text{-}\lim_n f^{(\omega)}_n \]
for $\mathbb{P}$-almost every $\omega$.
\end{theorem}

\begin{proof}
Define $\Omega^\prime$ as the intersection of three events:
\begin{align*}
\Omega^\prime & = \left\{ \omega\in \Omega : P_n^{(\omega)}\Rightarrow P \right\} \cap  \left\{ \omega\in \Omega : P_n^{(\omega)}(B(0,q)^c) \to P(B(0,q)^c) \; \forall q\in \mathbb{N} \right\} \\
 & \quad \quad \quad \quad \quad \cap  \left\{ \omega\in \Omega : \int_X \mathbb{I}_{B(0,q)^c}(x)M(\|x\|_X) \; P_n^{(\omega)} (\text{d} x) \to \int_X \mathbb{I}_{B(0,q)^c}(x) M(\|x\|_X) \; P(\text{d} x) \; \forall q\in \mathbb{N} \right\}.
\end{align*}
By the almost sure weak convergence of the empirical measure the
first of these events has probability one, the second and third are
characterized by the convergence of a countable collection of empirical averages to their
population average and, by the strong law of large numbers, each has probability
one. Hence $\mathbb{P}(\Omega^\prime)=1$.

Fix $\omega \in \Omega^\prime$: we will show that the lim inf inequality holds and
a recovery sequence exists for this $\omega$ and hence for every $\omega\in\Omega^\prime$.
We start by showing the lim inf inequality, allowing $\{\mu^n\}_{n=1}^\infty\in X^k$ to denote any sequence which converges weakly to $\mu\in X^k$.
We are required to show:
\[ \liminf_{n\to \infty} f_n^{(\omega)}(\mu^n) \geq  f_\infty (\mu). \]
By Theorem~1.1 in~\cite{feinberg14} we have
\[ \int_X \liminf_{n\to\infty, x^\prime\to x} g_{\mu^n}(x^\prime) \; P(\text{d} x) \leq \liminf_{n\to \infty} \int_X g_{\mu^n}(x) \; P_n^{(\omega)}(\text{d}x) = \liminf_{n\to\infty} P_n^{(\omega)} g_{\mu^n}. \]
For each $x\in X$, we have by Assumption~\ref{ass:d}.2 that
\[ \liminf_{x^\prime\to x,n\to \infty} d(x^\prime,\mu_j^n) \geq d(x,\mu_j). \]
By taking the minimum over $j$ we have
\[ \liminf_{x^\prime\to x,n\to \infty} g_{\mu^n}(x^\prime) = \bigwedge_{j=1}^k \liminf_{x^\prime\to x,n\to \infty} d(x^\prime,\mu_j^n) \geq \bigwedge_{j=1}^k d(x,\mu_j) = g_{\mu}(x). \]
Hence
\[ \liminf_{n\to \infty} f_n^{(\omega)}(\mu^n) = \liminf_{n\to \infty} P_n^{(\omega)} g_{\mu^n} \geq \int_X g_{\mu}(x) \; P(\text{d} x) = f_\infty(\mu) \]
as required.

We now establish the existence of a recovery sequence for every $\omega \in \Omega^\prime$ and every $\mu \in X^k$.
Let $\mu^n=\mu\in X^k$.
Let $\zeta_q$ be a $C^\infty(X)$ sequence of functions such that $0\leq \zeta_q(x) \leq 1$ for all $x\in X$, $\zeta_q(x) = 1$ for $x\in B(0,q-1)$ and $\zeta_q(x) = 0$ for $x\not\in B(0,q)$.
Then the function $\zeta_q(x)g_\mu(x)$ is continuous in $x$ (and with respect to convergence in $\|\cdot\|_X$) for all $q$.
We also have
\begin{align*}
\zeta_q(x) g_\mu(x) & \leq \zeta_q(x) d(x,\mu_1) \\
 & \leq \zeta_q(x) M(\|x-\mu_1\|_X) \\
 & \leq \zeta_q(x) M(\|x\|_X + \|\mu_1\|_X) \\
 & \leq M(q + \|\mu_1\|_X)
\end{align*}
so $\zeta_q g_\mu$ is a continuous and bounded function, hence by the weak convergence of $P_n^{(\omega)}$ to $P$ we have
\[ P_n^{(\omega)} \zeta_q g_\mu \to P \zeta_q g_\mu \]
as $n\to \infty$ for all $q\in\mathbb{N}$.
For all $q\in\mathbb{N}$ we have
\begin{align*}
\limsup_{n\to \infty} |P_n^{(\omega)} g_\mu - Pg_\mu | & \leq \limsup_{n\to \infty} |P_n^{(\omega)} g_\mu - P_n^{(\omega)}\zeta_q g_\mu | + \limsup_{n\to \infty} |P_n^{(\omega)} \zeta_q g_\mu - P \zeta_q g_\mu | + \limsup_{n\to \infty} |P \zeta_q g_\mu - Pg_\mu | \\
 & = \limsup_{n\to \infty} |P_n^{(\omega)} g_\mu - P_n^{(\omega)}\zeta_q g_\mu | + |P \zeta_q g_\mu - Pg_\mu |.
\end{align*}
Therefore,
\[ \limsup_{n\to \infty} |P_n^{(\omega)} g_\mu - Pg_\mu | \leq \liminf_{q\to \infty} \limsup_{n\to \infty} |P_n^{(\omega)} g_\mu - P_n^{(\omega)}\zeta_q g_\mu | \]
by the dominated convergence theorem.
We now show that the right hand side of the above expression is equal to zero.
We have
\begin{align*}
|P_n^{(\omega)} g_\mu - P_n^{(\omega)}\zeta_q g_\mu | & \leq P_n^{(\omega)} \mathbb{I}_{(B(0,q-1))^c} g_\mu \\
 & \leq P_n^{(\omega)} \mathbb{I}_{(B(0,q-1))^c} d(\cdot,\mu_1) \\
 & \leq P_n^{(\omega)} \mathbb{I}_{(B(0,q-1))^c} M(\|\cdot-\mu_1\|_X) \\
 & \leq \gamma \left(P_n^{(\omega)} \mathbb{I}_{(B(0,q-1))^c}M(\|\cdot\|_X) + M(\|\mu_1\|_X) P_n^{(\omega)} \mathbb{I}_{(B(0,q-1))^c} \right) \\
 & \to \gamma \left( P \mathbb{I}_{(B(0,q-1))^c}M(\|\cdot\|_X) + M(\|\mu_1\|_X) P \mathbb{I}_{(B(0,q-1))^c} \right) \quad \text{as } n\to \infty \\
 & \to 0 \quad \text{as } q \to \infty
\end{align*}
where the last limit follows by the monotone convergence theorem.
We have shown
\[ \lim_{n\to \infty} |P_n^{(\omega)} g_\mu - Pg_\mu | = 0. \]
Hence
\[ f_n^{(\omega)}(\mu) \to f_\infty(\mu) \]
as required.
\end{proof}

Now we have established almost sure $\Gamma$-convergence we establish the boundedness condition in Proposition~\ref{lem:bdd} so we can apply Theorem~\ref{thm:conmin}.

\begin{proposition}
\label{lem:bdd}
Assuming the conditions of Theorem~\ref{thm:gamcon} and define $\|\cdot\|_k$ by~\eqref{eq:cons:norm}, there exists $R>0$ such that
\[ \inf_{\mu\in X^k} f_n^{(\omega)}(\mu) = \inf_{\|\mu\|_k\leq R} f_n^{(\omega)}(\mu) \quad \forall n \text{ sufficiently large} \]
for $\mathbb{P}$-almost every $\omega$.
In particular $R$ is independent of $n$.
\end{proposition}

\begin{proof}
The structure of the proof is similar to \cite[Lemma 2.1]{lember03}.
We argue by contradiction.
In particular we argue that if a cluster center is unbounded then in the limit the minimum is achieved over the remaining $k-1$ cluster centers.
We then use Assumption~\ref{ass:d}.4 to imply that adding an extra cluster center will strictly decrease the minimum, and hence we have a contradiction.

We define $\Omega^{\prime\prime}$ to be
\[ \Omega^{\prime\prime} = \cap_{\delta\in\mathbb{Q}\cap(0,\infty), l=1,2,\dots,k} \left\{\omega\in\Omega^\prime: P_n^{(\omega)} (B(\mu^\dagger_l,\delta))\to P(B(\mu^\dagger_l,\delta)) \right\}. \]
As $\Omega^{\prime\prime}$ is the countable intersection of sets of probability one, we have $\mathbb{P}(\Omega^{\prime\prime})=1$.
Fix $\omega\in\Omega^{\prime\prime}$ and assume that the cluster centers $\mu^n\in X^k$ are almost minimizers, i.e.
\[ f_n^{(\omega)}(\mu^n) \leq \inf_{\mu\in X^k} f_n^{(\omega)}(\mu)+
\varepsilon_n \]
for some sequence $\varepsilon_n>0$ such that
\begin{equation} \label{nullseq}
\lim_{n\to \infty} \varepsilon_n=0.
\end{equation}

Assume that $\lim\limits_{n \to \infty}\|\mu^n\|_{k}= \infty$. There exists $l_n \in \{1,\ldots,k\}$ such that $\lim\limits_{n \to \infty}\|\mu^n_{l_n}\|_X= \infty$.
Fix $x\in X$ then
\[ d(x,\mu^n_{l_n}) \geq m(\|\mu^n_{l_n}-x\|_X) \to \infty. \]
Therefore, for each $x\in X$,
\[ \lim_{n \to \infty} \left(\bigwedge_{j=1}^k d(x,\mu^n_{j}) - \bigwedge_{j\neq l_n} d(x,\mu^n_j)\right)= 0. \]
Let $\delta>0$ then there exists $N$ such that for $n\geq N$
\[ \bigwedge_{j=1}^k d(x,\mu^n_{j}) - \bigwedge_{j\neq l_n} d(x,\mu^n_j) \geq -\delta. \]
Hence
\[ \liminf_{n\to \infty} \int \left( \bigwedge_{j=1}^k d(x,\mu^n_j) - \bigwedge_{j\neq l_n} d(x,\mu^n_{j}) \right) \; P_n^{(\omega)}(\text{d} x) \geq -\delta. \]
Letting $\delta\to 0$ we have
\[ \liminf_{n\to \infty} \int \left( \bigwedge_{j=1}^k d(x,\mu^n_j) - \bigwedge_{j\neq l_n} d(x,\mu^n_{j}) \right) \; P_n^{(\omega)}(\text{d} x) \geq 0 \]
and moreover
\begin{equation} \label{eq:contra}\liminf_{n\to \infty} \left( f_n^{(\omega)}\left(\mu^{n}\right) - f_n^{(\omega)}\left((\mu^n_j)_{j\neq l_n} \right)\right)\geq 0,
\end{equation}
where we interpret $f_n^{(\omega)}$ accordingly.
It suffices to demonstrate that
\begin{equation} \label{eq:difffnl}
\liminf_{n\to \infty}\left(\inf_{\mu \in X^k} f_n^{(\omega)}(\mu) - \inf_{\mu\in X^{k-1}} f_n^{(\omega)}(\mu)\right) <0.
\end{equation}
Indeed, if \eqref{eq:difffnl} holds, then
\begin{align*}
& \liminf_{n\to \infty} \left( f_{n}^{(\omega)}\left(\mu^{n}\right) - f_{n}^{(\omega)}\left((\mu^n_j)_{j\neq l_n} \right)\right) \\
= & \lim_{n\to \infty}\bigl(
\underbrace{ f_{n}^{(\omega)}\left(\mu^{n}\right) - \inf_{\mu \in X^{k}} f_n^{(\omega)}(\mu)}_{\leq \varepsilon_n}\bigr) + \liminf_{n\to \infty}\left(\inf_{\mu \in X^{k}} f_n^{(\omega)}(\mu)- f_{n}^{(\omega)}\left((\mu^n_j)_{j\neq l_n} \right)\right)\\
< & 0 \quad \text{by \eqref{nullseq} and \eqref{eq:difffnl}},
\end{align*}
but this contradicts \eqref{eq:contra}.

We now establish \eqref{eq:difffnl}.
By Assumption~\ref{ass:d}.4 there exists $k$ centers $\mu_j^\dagger\in X$ and $\delta_1>0$ such that $\min_{j\neq l} \|\mu^\dagger_j - \mu_l^\dagger\|_X \geq \delta_1$.
Hence for any $\mu\in X^{k-1}$ there exists $l\in\{1,2,\dots,k\}$ such that we have
\[ \| \mu^\dagger_l - \mu_j \|_X \geq \frac{\delta_1}{2} \quad \quad \text{for } j=1,2,\dots,k-1. \]
Proceeding with this choice of $l$, for $x\in B(\mu^\dagger_l,\delta_2)$ (for any $\delta_2 \in (0,\delta_1/2)$) we have
\[ \|\mu_j - x\|_X \geq \frac{\delta_1}{2} - \delta_2 \]
and therefore $d(\mu_j,x) \geq m(\frac{\delta_1}{2}-\delta_2)$ for all $j=1,2,\dots,k-1$.
Also
\begin{equation} \label{minmax}
D_l(\mu) := \min_{j=1,2,\dots, k-1} d(x,\mu_j) - d(x,\mu_l^\dagger) \geq m(\frac{\delta_1}{2}-\delta_2) - M(\delta_2).
\end{equation}
So for $\delta_2$ sufficiently small there exists $\epsilon>0$ such that
\[ D_l(\mu) \geq \epsilon. \]
Since the right hand side is independent of $\mu\in X^{k-1}$,
\[ \inf_{\mu\in X^{k-1}} \max_l D_l(\mu) \geq \epsilon. \]

Define the characteristic function
\[ \chi_\mu(\xi)=\begin{cases} 1 & \text{ if } \|\xi-\mu_{l(\mu)}^\dagger\|_X < \delta_2\\
0 & \text{ otherwise,}\end{cases} \]
where $l(\mu)$ is the maximizer in (\ref{minmax}).
For each $\omega\in\Omega^{\prime\prime}$ one obtains
\begin{align*}
\inf_{\mu\in X^{k-1}} f_n^{(\omega)}(\mu) & = \inf_{\mu\in X^{k-1}} \frac{1}{n} \sum_{i=1}^n \bigwedge_{j=1}^{k-1} d(\xi_i,\mu_j) \\
 & \geq \inf_{\mu\in X^{k-1}} \frac{1}{n} \sum_{i=1}^n \left[ \bigwedge_{j=1}^{k-1} d(\xi_i,\mu_j)\left(1-\chi_\mu(\xi_i)\right) + \left( d(\xi_i,\mu^\dagger_{l(\mu)}) + \epsilon \right) \chi_\mu(\xi_i) \right] \\
 & \geq \inf_{\mu\in X^k} f_n^{(\omega)}(\mu) + \epsilon \min_{l=1,2,\dots,k} P_n^{(\omega)}(B(\mu^\dagger_l,\delta_2)).
\end{align*}
Then since $P_n^{(\omega)}(B(\mu^\dagger_l,\delta_2))\to P(B(\mu^\dagger_l,\delta_2))>0$ by Assumption~\ref{ass:d}.4 (for $\delta_2\in\mathbb{Q}\cap (0,\infty)$) we can conclude~\eqref{eq:difffnl} holds.
\end{proof}

\begin{remark}
\label{rem:cons:minexist}
One can easily show that Assumption~\ref{ass:d}.2 implies that $d$ is weakly lower semi-continuous in its second argument which carries through to $f_n^{(\omega)}$.
It follows that on any bounded (or equivalently as $X$ is reflexive: weakly compact) set the infimum of $f_n^{(\omega)}$ is achieved.
Hence the infimum in Proposition~\ref{lem:bdd} is actually a minimum.
\end{remark}

We now easily prove convergence by application of Theorem~\ref{thm:conmin}.

\begin{theorem}
\label{thm:cons}
Assuming the conditions of Theorem~\ref{thm:gamcon} and Proposition~\ref{lem:bdd} the minimization problem associated with the $k$-means method converges.
I.e. for $\mathbb{P}$-almost every $\omega$:
\[ \min_{\mu\in X^k} f_\infty (\mu) = \lim_{n \to \infty} \min _{\mu\in X^k} f_n^{(\omega)} (\mu). \]
Furthermore any sequence of minimizers $\mu^n$ of $f_n^{(\omega)}$ is almost surely weakly precompact and any weak limit point minimizes $f_\infty$.
\end{theorem}

\section{The Case of General \texorpdfstring{$Y$}{Y} \label{sec:MTT}}

In the previous section the data, $\xi_i$, and cluster centers, $\mu_j$, took their
values in a common space, $X$. We now remove this restriction and let $\xi_i:\Omega \rightarrow X$ and $\mu_j\in Y$.
We may want to use this framework to deal with finite dimensional data and
infinite dimensional cluster centers, which can lead to the variational
problem having uninformative minimizers.

In the previous section the cost function $d$ was assumed to scale with the underlying norm.
This is no longer appropriate when $d:X\times Y\to[0,\infty)$.
In particular if we consider the smoothing-data association problem then the natural choice of $d$ is a pointwise distance which will lead to the optimal cluster centers interpolating data points.
Hence, in any $H^s$ norm with $s\geq 1$, the optimal cluster centers ``blow up''.

One possible solution would be to weaken the space to $L^2$ and allow this type of behavior.
This is undesirable from both modeling and mathematical perspectives:
If we first consider the modeling point of view then we do not expect our estimate to perfectly fit the data which is observed in the presence of noise.
It is natural that the cluster centers are smoother than the data alone would suggest.
It is desirable that the optimal clusters should reflect reality.
From the mathematical point of view, restricting ourselves to only very weak spaces gives no hope of obtaining a strongly convergent subsequence.

An alternative approach is, as is common in the smoothing literature, to use a regularization term.
This approach is also standard when dealing with ill-posed inverse problems.
This changes the nature of the problem and so requires some justification.
In particular the scaling of the regularization with the data is of fundamental importance.
In the following section we argue that scaling motivated by a simple Bayesian interpretation of the problem is not strong enough (unsurprisingly, countable collections of finite dimensional observations do not carry enough information to provide consistency when dealing with infinite dimensional parameters).
In the form of a simple example we show that the optimal cluster center is unbounded in the large data limit when the regularization goes to zero sufficiently quickly.
The natural scaling in this example is for the regularization to vary with the number of observations as $n^p$ for $p\in[-\frac{4}{5},0]$.
We consider the case $p=0$ in Section \ref{sec:MTT:theory}.
This type of regularization is understood as penalized likelihood estimation \cite{good71}.

Although it may seem undesirable for the limiting problem to depend upon the regularization it is unavoidable in ill-posed problems such as this one: there is not sufficient information, in even countably infinite collections of observations to recover the unknown cluster centers and exploiting known (or expected) regularity in these solutions provides one way to combine observations with qualitative prior beliefs about the cluster centers in a principled manner.
There are many precedents for this approach, including \cite{hall05} in which the consistency of penalized splines is studied using, what in this paper we call, the $\Gamma$-limit.
In that paper a fixed regularization was used to define the limiting problem in order to derive an estimator.
Naturally, regularization strong enough to alter the limiting problem influences the solution and we cannot hope to obtain consistent estimation in this setting, even in settings in which the cost function can be interpreted as the log likelihood of the data generating process.
In the setting of \cite{hall05}, the regularization is finally scaled to zero whereupon under assumptions the estimator converges to the truth but such a step is not feasible in the more complicated settings considered here.

When more structure is available it may be desirable to further investigate the regularization.
For example with $k=1$ the non-parametric regression model is equivalent to the white noise model \cite{brown96} for which optimal scaling of the regularization is known~\cite{agapiou13, zhao00}.
It is the subject of further work to extend these results to $k>1$.

With our redefined $k$-means type problem we can replicate the results of the previous section, and do so in Theorem~\ref{thm:kmeanscons}.
That is, we prove that the $k$-means method converges where $Y$ is a general separable and reflexive Banach space and in particular need not be equal to $X$.

This section is split into three subsections.
In the first we motivate the regularization term.
The second contains the convergence theory in a general setting.
Establishing that the assumptions of this subsection hold is non-trivial and
so, in the third subsection, we show an application to the smoothing-data
association problem.

\subsection{Regularization \label{sec:MTT:reg}}

In this section we use a toy, $k=1$, smoothing problem to motivate an approach to regularization which is adopted in what follows.
We assume that the cluster centers are periodic with equally spaced observations so we may use a Fourier argument.
In particular we work on the space of 1-periodic functions in $H^2$,
\begin{equation} \label{eq:MTT:reg:Y}
Y = \left\{ \mu:[0,1]\to \mathbb{R} \text{ s.t. } \mu(0)=\mu(1) \text{ and } \mu\in H^2 \right\}.
\end{equation}

For arbitrary sequences $(a_n)$, $(b_n)$ and data $\Psi_n=\{(t_j,z_j)\}_{j=1}^n \subset [0,1] \times \mathbb{R}^d$ we define the functional
\begin{equation} \label{eq:5.1fn}
f_n^{(\omega)}(\mu) = a_n \sum_{j=0}^{n-1} \left|\mu(t_j) - z_j\right|^2 + b_n \|\partial^2 \mu\|^2_{L^2}.
\end{equation}
Data are points in space-time: $[0,1]\times \mathbb{R}$.
The regularization is chosen so that it penalizes the $L^2$ norm of the second derivative.
For simplicity, we employ deterministic measurement times $t_j$ in the following proposition although this lies outside the formal framework which we consider subsequently.
Another simplification we make is to use convergence in expectation rather than almost sure convergence.
This simplifies our arguments.
We stress that this section is the motivation for the problem studied in Section~\ref{sec:MTT:theory}.
We will give conditions on the scaling of $a_n$ and $b_n$ that determine whether $\mathbb{E}\min f_n^{(\omega)}$ and $\mathbb{E}\mu^n$ stay bounded where $\mu^n$ is the minimizer of $f_n^{(\omega)}$.

\begin{proposition} \label{prop:boundedmin}
Let data be given by $\Psi_n=\{(t_j,z_j)\}_{j=1}^n$ with $t_j=\frac{j}{n}$ under the assumption $z_j = \mu^\dagger(t_j) + \epsilon_j$ for $\epsilon_j$ iid noise with finite variance and $\mu^\dagger\in L^2$ and define $Y$ by~\eqref{eq:MTT:reg:Y}.
Then $\inf_{\mu\in Y} f_n^{(\omega)}(\mu)$ defined by \eqref{eq:5.1fn} stays bounded (in expectation) if $a_n=O(\frac{1}{n})$ for any positive sequence $b_n$.
\end{proposition}

\begin{proof}
Assume $n$ is odd.
Both $\mu$ and $z$ are 1-periodic so we can write
\[ \mu(t) = \frac{1}{n}\sum_{l=-\frac{n-1}{2}}^{\frac{n-1}{2}} \hat{\mu}_l e^{2\pi i lt} \quad \quad \text{and} \quad \quad z_j = \frac{1}{n}\sum_{l=-\frac{n-1}{2}}^{\frac{n-1}{2}} \hat{z}_l e^{\frac{2\pi i lj}{n}} \]
with
\[ \hat{\mu}_l = \sum_{j=0}^{n-1} \mu(t_j) e^{-\frac{2\pi i lj}{n}} \quad \quad \text{and} \quad \quad \hat{z}_l = \sum_{j=0}^{n-1} z_j e^{-\frac{2\pi i lj}{n}}. \]
We will continue to use the notation that $\hat{\mu}_l$ is the Fourier transform of $\mu$.
We write
\[ \hat{\mu}:=\left(\hat{\mu}_{-\frac{n-1}{2}}, \hat{\mu}_{-\frac{n-1}{2}+1}, \dots, \hat{\mu}_{\frac{n-1}{2}} \right). \]
Similarly for $z$.

Substituting the Fourier expansion of $\mu$ and $z$ into $f_n^{(\omega)}$ implies
\[ f_n^{(\omega)}(\mu) = \frac{a_n}{n} \left( \langle \hat{\mu},\hat{\mu} \rangle - 2 \langle \hat{\mu},\hat{z}\rangle + \langle \hat{z},\hat{z} \rangle + \frac{\gamma_n}{n} \langle l^4\hat{\mu},\hat{\mu}\rangle \right) \]
where $\gamma_n = \frac{16\pi^4 b_n}{a_n}$ and $\langle \hat{x},\hat{z}\rangle = \sum_l \hat{x}_l \overline{\hat{z}}_l$.
The Gateaux derivative $\partial f_n^{(\omega)}(\mu;\nu)$ of $f_n^{(\omega)}$ at $\mu$ in the direction $\nu$ is
\[ \partial f_n^{(\omega)}(\mu;\nu) = \frac{2a_n}{n} \left\langle \hat{\mu} - \hat{z} + \frac{\gamma_n l^4}{n} \hat{\mu},\hat{\nu} \right\rangle. \]
Which implies the minimizer $\mu^n$ of $f_n^{(\omega)}$ is (in terms of its Fourier expansion)
\[ \hat{\mu}^n_l = \left(1+\frac{\gamma_n l^4}{n} \right)^{-1} \hat{z} := \left( \left(1+\frac{\gamma_n l^4}{n} \right)^{-1} \hat{z}_l \right)_{l=-\frac{n-1}{2}}^{\frac{n-1}{2}}. \]
It follows that the minimum is
\[ \mathbb{E}\left(f_n^{(\omega)}(\mu^n)\right) = \frac{a_n}{n} \mathbb{E}\left(\left\langle \left( 1 + \frac{n}{\gamma_n l^4} \right)^{-1} \hat{z},\hat{z} \right\rangle \right) \leq a_n \sum_{j=0}^{n-1} \mathbb{E} z_j^2 \lesssim 2a_n n \left( \|\mu^\dagger\|_{L^2}^2 + \text{Var}(\epsilon)  \right). \]
Similar expressions can be obtained for the case of even $n$.
\end{proof}

Clearly the natural choice for $a_n$ is
\[ a_n = \frac{1}{n} \]
which we use from here.
We let $b_n = \lambda n^p$ and therefore $\gamma_n = 16\pi^4 \lambda n^{p+1}$.
From Proposition~\ref{prop:boundedmin} we immediately have $\mathbb{E} \min f_n^{(\omega)}$ is bounded for any choice of $p$.
In our next proposition we show that for $p\in[-\frac{4}{5},0]$ our minimizer is bounded in $H^2$ whilst outside this window the norm either blows up or the second derivative converges to zero.
For simplicity in the calculations we impose the further condition that $\mu^\dagger(t)=0$.

\begin{proposition}
In addition to the assumptions of Proposition \ref{prop:boundedmin} let $a_n = \frac{1}{n}$, $b_n=\lambda n^p$, $\epsilon_j \iid N(0,\sigma^2)$ and assume that $\mu^n$ is the minimizer of $f_n^{(\omega)}$.
\begin{itemize}
\item[1.] For $n$ sufficiently large there exists $M_1>0$ such that for all $p$ and $n$ the $L^2$ norm is bounded:
\[ \mathbb{E} \|\mu^n\|^2_{L^2} \leq M_1. \]
\item[2.] If $p>0$ then
\[ \mathbb{E} \|\partial^2 \mu^n\|^2_{L^2} \to 0 \quad \text{as } n\to \infty. \]
\end{itemize}
If we further assume that $\mu^\dagger(t)=0$, then the following statements are true.
\begin{itemize}
\item[3.] For all $p \in [-\frac{4}{5},0]$ there exists $M_2>0$ such that
\[ \mathbb{E}\|\partial^2 \mu^n\|^2_{L^2} \leq M_2. \]
\item[4.] If $p<-\frac{4}{5}$ then
\[ \mathbb{E}\|\partial^2 \mu^n\|^2_{L^2} \to \infty \quad \text{as } n\to \infty. \]
\end{itemize}
\end{proposition}

\begin{proof}
The first two statements follow from
\begin{align*}
\mathbb{E} \|\mu^n\|_{L^2}^2 & \lesssim 2\left( \|\mu^\dagger\|_{L^2}^2 + \text{Var}(\epsilon) \right) \\
\mathbb{E} \|\partial^2\mu^n\|_{L^2}^2 & \lesssim \frac{8\pi^4 n}{\gamma_n} \left( \|\mu^\dagger\|^2_{L^2} + \text{Var}(\epsilon) \right)
\end{align*}
which are easily shown.
Statement 3 is shown after statement 4.

Following the calculation in the proof of Proposition~\ref{prop:boundedmin}, and assuming that $\mu^\dagger(t)=0$, it is easily shown that
\begin{align} \label{exprep} \mathbb{E}\|\partial^2 \mu^n\|_{L^2}^2 = \frac{16\pi^4\sigma^2}{n} \sum_{l=-\frac{n-1}{2}}^{\frac{n-1}{2}} \frac{l^4}{(1+16\pi^4\lambda n^p l^4)^2} =: S(n)
\end{align}
since $\mathbb{E}|\hat{z}_l|^2 = \sigma^2n$.
To show $S(n)\to \infty$ we will manipulate the Riemann sum approximation of
\[ \int_{-\frac{1}{2}}^{\frac{1}{2}} \frac{x^4}{(1+16\pi^4\lambda x^4)^2} \; \text{d} x = C \]
where $0<C<\infty$.
We have
\begin{align*}
\int_{-\frac{1}{2}}^{\frac{1}{2}} \frac{x^4}{(1+ 16\pi^4\lambda x^4)^2} \; \text{d} x & = n^{1+\frac{p}{4}} \int_{-\frac{1}{2} n^{-1-\frac{p}{4}}}^{\frac{1}{2} n^{-1-\frac{p}{4}}} \frac{n^{4+p}w^4}{(1+16\pi^4\lambda n^{4+p}w^4)^2} \; \text{d} w \quad \text{where } x = n^{1+\frac{p}{4}} w \\
 & \approx n^{\frac{5p}{4}} \sum_{l=-\left\lfloor\frac{1}{2} n^{-\frac{p}{4}}\right\rfloor}^{\left\lfloor\frac{1}{2} n^{-\frac{p}{4}}\right\rfloor} \frac{l^4}{(1+16\pi^4 \lambda n^pl^4)^2} =: R(n).
\end{align*}
Therefore assuming $p>-4$ we have
\[ S(n) \geq \frac{16\pi^4 \sigma^2}{n^{1+\frac{5p}{4}}} R(n). \]
So for $1+\frac{5p}{4} <0$ we have $S(n)\to \infty$.
Since $S(n)$ is monotonic in $p$ then $S(n)\to \infty$ for all $p<-\frac{4}{5}$.
This shows that statement 4 is true.

Finally we establish the third statement.
If $p = -\frac{4}{5}$ then
\begin{align*}
S(n) & = 16\pi^4\sigma^2 R(n) + \frac{16\pi^4\sigma^2}{n}\left( \sum_{l=-\frac{n-1}{2}}^{\lfloor\frac{n^{\frac{1}{5}}}{2}\rfloor -1} \frac{l^4}{(1+16\pi^4\lambda n^p l^4)^2} + \sum_{l=\lfloor\frac{n^{\frac{1}{5}}}{2}\rfloor +1}^{\frac{n-1}{2}} \frac{l^4}{(1+16\pi^4\lambda n^p l^4)^2} \right) \\
 & \leq 16\pi^4\sigma^2 R(n) + \frac{2\pi^4\sigma^2}{n^{\frac{1}{5}}(1+\pi^4\lambda)^2}.
\end{align*}
The remaining cases $p \in [-\frac{4}{5},0]$ are a consequence of (\ref{exprep}) which implies that $p \mapsto \mathbb{E}(\partial^2 \mu)$ is non-increasing.
\end{proof}

By the Poincar\'e inequality it follows that if $p\geq -\frac{4}{5}$ then the $H^2$ norm of our minimizer stays bounded as $n\to \infty$.
Our final calculation in this section is to show that the regularization for $p\in [-\frac{4}{5},0]$ is not too strong.
We have already shown that $\|\partial^2\mu^n\|_{L^2}$ is bounded (in expectation) in this case but we wish to make sure that we don't have the stronger result that $\|\partial^2\mu^n\|_{L^2} \to 0$.

\begin{proposition}
With the assumptions of Proposition \ref{prop:boundedmin} and $a_n=\frac{1}{n}$, $b_n=\lambda n^p$ with $p\in [-\frac{4}{5},0]$ there exists a choice of $\mu^\dagger$ and a constant $M>0$ such that if $\mu^n$ is the minimizer of $f_n^{(\omega)}$ then
\begin{equation} \label{eq:5.1optbound}
\mathbb{E} \|\partial^2 \mu^n\|^2_{L^2} \geq M.
\end{equation}
\end{proposition}
\begin{proof}
We only need to prove the proposition for $p=0$ (the strongest regularization) and find one $\mu^\dagger$ such that \eqref{eq:5.1optbound} is true.
Let $\mu^\dagger(t) = 2\cos(2\pi t) = e^{2\pi i t} + e^{-2\pi i t}$.
Then the Fourier transform of $\mu^\dagger$ satisfies $\hat{\mu}^\dagger_l=0$ for $l\neq \pm 1$ and $\hat{\mu}^\dagger_l=n$ for $l=\pm 1$.
So,
\begin{align*}
\mathbb{E} \|\partial^2 \mu^n\|_{L^2}^2 & = \frac{16\pi^4}{n^2} \sum_{l=-\frac{n-1}{2}}^{\frac{n-1}{2}} \frac{l^4}{(1+16\pi^4\lambda l^4)^2} \mathbb{E}|\hat{z}_l|^2 \\
 & \gtrsim \frac{16\pi^4}{n^2} \sum_{l=-\frac{n-1}{2}}^{\frac{n-1}{2}} \frac{l^4}{(1+16\pi^4\lambda l^4)^2} |\hat{\mu}^\dagger_l|^2 \\
 & = \frac{32\pi^4}{(1+16\pi^4\lambda)^2} > 0.
\end{align*}
\end{proof}

We have shown that the minimizer is bounded for any $p\geq -\frac{4}{5}$ and $\|\partial^2\mu^n\|_{L^2} \to 0$ for $p>0$.
The case $p>0$ is clearly undesirable as we would be restricting ourselves to straight lines.
The natural scaling for this problem is in the range $p\in [-\frac{4}{5},0]$.
In the remainder of this paper we consider the case $p=0$.
This has the advantage that, not only $\mathbb{E}\|\partial^2 \mu^n\|_{L^2}$, but also
$\mathbb{E}f_n^{(\omega)}(\mu^n)$ is $O(1)$ as $n \to \infty$.
In fact we will show that with this choice of regularization we do not need to choose $k$ dependent on the data generating model.
The regularization makes the methodology sufficiently robust to have convergence even for poor choices of $k$.
For example, if there exists a data generating process which is formed of a $k^\dagger$-mixture model then for our method to be robust does not require us to choose $k=k^\dagger$.
Of course with the `wrong' choice of $k$ the results may be physically meaningless and we should take care in how to interpret the results.
The point to stress is that the methodology does not rely on a data generating model.

The disadvantage of this is to potentially increase the bias in the method.
Since the $k$-means is already biased we believe the advantages of our approach outweigh the disadvantages.
In particular we have in mind applications where only a coarse estimate is needed.
For example the $k$-means method may be used to initialize some other algorithm.
Another application could be part of a decision making process: in Section \ref{sec:example1} we show the $k$-means methodology can be used to determine whether two tracks have crossed.

\subsection{Convergence For General \texorpdfstring{$Y$}{Y} \label{sec:MTT:theory}}

Let $(X,\|\cdot\|_X)$, $(Y,\|\cdot\|_Y)$ be reflexive, separable Banach spaces
%Let $(X,\|\cdot\|_X)$, $(Y,\|\cdot\|_Y)$ be Banach spaces where $X$ and $Y$ are reflexive and separable.
We will also assume that the data points, $\Psi_n =\{\xi_i\}_{i=1}^n \subset
X$ for $i=1,2,\dots,n$ are iid random elements with common law $P$.
As before $\mu=(\mu_1,\mu_2,\dots,\mu_k)$ but now the cluster centers $\mu_j\in Y$ for each $j$.
The cost function is $d:X\times Y\to [0,\infty)$.

The energy functions associated with the $k$-means algorithm in this setting
are slightly different to those used previously:
\begin{align}
g_{\mu}: X & \to \mathbb{R},\quad
g_{\mu}(x) = \bigwedge_{j=1}^k d(x,\mu_j), \notag \\
f_n^{(\omega)}: Y^k & \to \mathbb{R},\quad
f_n^{(\omega)}(\mu) = P_n^{(\omega)}g_{\mu} + \lambda r(\mu), \label{eq:fnreg} \\
f_\infty: Y^k & \to \mathbb{R}, \quad
f_\infty(\mu) = Pg_{\mu} + \lambda r(\mu). \label{eq:finftyreg}
\end{align}
The aim of this section is to show the convergence result:
\[ \hat{\theta}_n^{(\omega)} = \inf_{\mu\in Y^k} f_n^{(\omega)}(\mu) \to \inf_{\mu\in Y^k} f_\infty(\mu) = \theta \quad \text{and} \quad \text{as } n\to \infty \text{ for } \mathbb{P}\text{-almost every } \omega \]
and that minimizers converge (almost surely).

The key assumptions are given in Assumptions~\ref{ass:dandr}; they imply that $f_n^{(\omega)}$ is weakly lower semi-continuous and coercive.
In particular, Assumption~\ref{ass:dandr}.2 allows us to prove the lim inf inequality as we did for Theorem~\ref{thm:gamcon}.
Assumption~\ref{ass:dandr}.1 is likely to mean that our convergence results are limited to the case of bounded noise.
In fact, when applying the problem to the smoothing-data association problem, it is necessary to bound the noise in order for Assumption~\ref{ass:dandr}.5 to hold.
Assumption~\ref{ass:dandr}.5 implies that $f_n^{(\omega)}$ is (uniformly) coercive and hence allows us to easily bound the set of minimizers.
%It is the objective of a future paper to extend the convergence results given here to the smoothing-data association problem without assuming bounded noise.
It is the subject of ongoing research to extend the convergence results to unbounded noise for the smoothing-data association problem.
Assumption~\ref{ass:dandr}.3 is a measurability condition we require in order
to integrate and the weak lower semi-continuity of $r$ is needed for the to
obtain the lim inf inequality in the $\Gamma$-convergence proof.

We note that, since $Pd(\cdot,\mu_1)\leq \sup_{x\in \mathrm{supp}(P)} d(x,\mu_1) <\infty$, we have $f_\infty(\mu)<\infty$ for every $\mu\in Y^k$ (and since $r(\mu)<\infty$ for each $\mu\in Y^k$).

\begin{assumptions}
\label{ass:dandr}
We have the following assumptions on $d:X\times Y\to [0,\infty)$, $r:Y^k\to [0,\infty)$ and $P$.
\begin{itemize}
\item[2.1.] For all $y\in Y$ we have $\sup_{x\in \mathrm{supp}(P)} d(x,y)<\infty$ where $\mathrm{supp}(P)\subseteq X$ is the support of $P$.
\item[2.2.] For each $x\in X$ and $y\in Y$ we have that if $x_m\to x$ and $y_n\rightharpoonup y$ as $n,m\to \infty$ then
\[ \liminf_{n,m\to \infty} d(x_m,y_n) \geq d(x,y) \quad \text{and} \quad \lim_{m\to\infty} d(x_m,y) = d(x,y). \]
\item[2.3.] For every $y \in Y$ we have that $d(\cdot,y)$ is $\mathcal{X}$-measurable.
\item[2.4.] $r$ is weakly lower semi-continuous.
\item[2.5.] $r$ is coercive.
\end{itemize}
\end{assumptions}

We will follow the structure of Section~\ref{sec:cons}.
We start by showing that under the above conditions $f_n^{(\omega)}$ $\Gamma$-converges to $f_\infty$.
We then show that the regularization term guarantees that the minimizers to $f_n^{(\omega)}$ lie in a bounded set.
An application of Theorem~\ref{thm:conmin} gives the desired convergence result.
Since we were able to restrict our analysis to a weakly compact subset of $Y$ we are easily able to deduce the existence of a weakly convergent subsequence.

Similarly to the previous section on the product space $Y^k$ we use the norm $\|\mu\|_k:=\max_j \|\mu_j\|_{Y}$.

\begin{theorem}
\label{thm:gamcondiffspace}
Let $(X,\|\cdot\|_X)$ and $(Y,\|\cdot\|_Y)$ be separable and reflexive Banach spaces.
Assume $r:Y^k\to [0,\infty)$, $d:X\times Y\to [0,\infty)$ and the probability measure $P$ on $(X,\mathcal{X})$ satisfy the conditions in Assumptions~\ref{ass:dandr}.
For independent samples $\{\xi_i^{\omega)}\}_{i=1}^n$ from $P$ define $P_n^{(\omega)}$ to be the
empirical measure and $f_n^{(\omega)}:Y^k \to \mathbb{R}$ and $f_\infty:Y^k \to
\mathbb{R}$ by \eqref{eq:fnreg} and \eqref{eq:finftyreg} respectively and where $\lambda>0$.
Then
\[ f_\infty = \Gamma\text{-}\lim_n f_n^{(\omega)} \]
for $\mathbb{P}$-almost every $\omega$.
\end{theorem}

\begin{proof}
% Version before my changes, AMJ.
Define
\[ \Omega^\prime = \left\{ \omega\in\Omega : P_n^{(\omega)} \Rightarrow P \right\} \cap \left\{ \omega \in \Omega : \xi^{(\omega)}_i \in \text{supp}(P) \; \forall i\in \mathbb{N} \right\}. \]
Then $\mathbb{P}(\Omega^\prime)=1$.
For the remainder of the proof we consider an arbitrary $\omega\in\Omega^\prime$.
We start with the lim inf inequality.
Let $\mu^n\rightharpoonup \mu$ then
\[ \liminf_{n\to \infty} f_n^{(\omega)}(\mu^n) \geq f_\infty(\mu) \]
follows (as in the proof of Theorem~\ref{thm:gamcon}) by applying Theorem~1.1 in~\cite{feinberg14} and the fact that $r$ is weakly lower semi-continuous.

We now establish the existence of a recovery sequence.
Let $\mu\in Y^k$ and let $\mu^n=\mu$.
We want to show
\[ \lim_{n\to \infty} f_n^{(\omega)}(\mu) = \lim_{n\to \infty} P_n^{(\omega)}g_\mu + \lambda r(\mu) = Pg_\mu + \lambda r(\mu) = f_\infty(\mu). \]
Clearly this is equivalent to showing that
\[ \lim_{n\to \infty} P_n^{(\omega)} g_\mu = Pg_\mu. \]
Now $g_\mu$ are continuous by assumption on $d$.
Let $M=\sup_{x\in \text{supp}(P)}d(x,\mu_1)<\infty$ and note that $g_\mu(x) \leq M$ for all $x\in\text{supp}(P)$ and therefore bounded. % and $X_M = \left\{x\in X: d(x,\mu_1)\leq M\right\}$. % and set  so for $P$-almost every $x$
% 
%We define $\hat{g}(x) = g_\mu(x)$ for $x\in X_M$ and $\hat{g}(x) = 0$ otherwise.
%Since $P(X_M)=1$ we have that
%Since the integral is only defined upto equivalence classes then without loss of generality we can assume that
%It follows that $g_\mu$ is bounded on the support of $P$:
%\[ g_\mu(x) \leq d(x,\mu_1) \leq M \]
%for all $x\in \text{supp}(P)$.
%We can define a function $\hat{g}$ on the whole of $X$ so that $\hat{g}(x) = g_\mu(x)$ for $x\in \text{supp}(P)$ and $\hat{g}$ is bounded and continuous.
%Then we can clearly define a function $\hat{g}$ on the whole of $X$ such that $\hat{g}(x) = g_\mu(x)$ for $P$-almost every $x\in X$, $\hat{g}$ is continuous and $\hat{g}(x)\leq M$ for all $x\in X$.
%Therefore
Hence $P_n^{(\omega)} g_\mu \to Pg_\mu$. % since $P_n^{(\omega)}\Rightarrow Pg_\mu$. % then we have $P_n^{(\omega)} g_\mu \to Pg_\mu$ as required.
\end{proof}

\begin{proposition}
\label{lem:minset}
Assuming the conditions of Theorem~\ref{thm:gamcondiffspace}, then for $\mathbb{P}$-almost every $\omega$ there exists $N<\infty$ and $R>0$ such that
\[ \min_{\mu\in Y^k} f_n^{(\omega)}(\mu) = \min_{\|\mu\|_k \leq R} f_n^{(\omega)}(\mu) < \inf_{\|\mu\|_k > R} f_n^{(\omega)}(\mu) \quad \forall n\geq N. \]
In particular $R$ is independent of $n$.
\end{proposition}

\begin{proof}
%Pick $\nu = (\nu_1,\nu_1,\dots,\nu_1)\in Y^k$  and
Let
\[ \Omega^{\prime\prime} = \left\{ \omega\in\Omega^\prime : P_n^{(\omega)} \Rightarrow P \right\} \cap \left\{ \omega\in\Omega^\prime : P^{(\omega)}_n d(\cdot,0)\to Pd(\cdot,0) \right\}. \]
Then, for every $\omega \in \Omega^{\prime\prime}$, $f_n^{(\omega)}(0)\to f_\infty(0)<\infty$ where with a slight abuse of notation we denote the zero element in both $Y$ and $Y^k$ by $0$.
Take $N$ sufficiently large so that
\[ f_n^{(\omega)}(0) \leq f_\infty(0) + 1 \quad \quad \text{for all } n\geq N. \]
Then $\min_{\mu\in Y^k} f_n^{(\omega)}(\mu) \leq f_\infty(0) + 1$ for all $n\geq N$.
By coercivity of $r$ there exists $R$ such that if $\|\mu\|_k>R$ then $\lambda r(\mu) \geq f_\infty(0)+1$.
Therefore any such $\mu$ is not a minimizer and in particular any minimizer must be contained in the set $\left\{ \mu\in Y^k : \|\mu\|_k\leq R \right\}$.
%Hence $f_n^{(\omega)}(\mu)\geq f_n^{(\omega)}(\nu) +1$.
%Any minimizer $\mu^n$ of $f_n^{(\omega)}$ must therefore satisfy $\|\mu^n\|_k\leq R$.
\end{proof}

The convergence results now follows by applying Theorem~\ref{thm:gamcondiffspace} and Proposition~\ref{lem:minset} to Theorem~\ref{thm:conmin}.

\begin{theorem}
\label{thm:kmeanscons}
Assuming the conditions of Theorem~\ref{thm:gamcondiffspace} and Proposition~\ref{lem:minset} the minimization problem associated with the $k$-means method converges in the following sense:
\[ \min_{\mu\in Y^k} f_\infty(\mu) = \lim_{n\to \infty} \min_{\mu\in Y^k} f_n^{(\omega)}(\mu) \]
for $\mathbb{P}$-almost every $\omega$.
%Furthermore, for $\mathbb{P}$-almost every $\omega$, any sequence of minimizers $\mu^n$ of $f_n^{(\omega)}$ contains a weakly converging subsequence $\mu^{n_m}\rightharpoonup \mu^\infty$ where $\mu^\infty$ minimizes $f_\infty$.
Furthermore any sequence of minimizers $\mu^n$ of $f_n^{(\omega)}$ is almost surely weakly precompact and any weak limit point minimizes $f_\infty$.
\end{theorem}

It was not necessary to assume that cluster centers are in a common space.
A trivial generalization would allow each $\mu_j\in Y^{(j)}$ with the cost and regularization terms appropriately defined; in this setting Theorem~\ref{thm:kmeanscons} holds.

\subsection{Application to the Smoothing-Data Association Problem \label{sec:MTT:app}}

In this section we give an application to the smoothing-data association problem and show the assumptions in the previous section are met.
For $k=1$ the smoothing-data association problem is the problem of fitting a curve to a data set (no data association).
For $k>1$ we couple the smoothing problem with a data association problem.
Each data point is associated with an unknown member of a collection of $k$ curves.
Solving the problem involves simultaneously estimating both the data partition (i.e. the association of observations to curves) and the curve which best fits each subset of the data.
By treating the curve of best fit as the cluster center we are able to approach this problem using the $k$-means methodology.
The data points are points in space-time whilst cluster centers are functions from time to space.

We let the Euclidean norm on $\mathbb{R}^\kappa$ be given by $|\cdot|$.
Let $X=\mathbb{R} \times \mathbb{R}^\kappa$ be the data space.
We will subsequently assume that the support of $P$, the common law of our observations, is contained within $\tilde{X} = [0,T]\times X^\prime$ where $X^\prime\subseteq [-\tilde{N},\tilde{N}]^\kappa$.
We define the cluster center space to be $Y=H^2([0,T])$, the Sobolev space of functions from $[0,T]$ to $\mathbb{R}^\kappa$.
Clearly $X$ and $Y$ are separable and reflexive.
The cost function $d:X\times Y\to [0,\infty)$ is defined by
\begin{equation} \label{eq:5.2d}
d(\xi,\mu_j) = |z-\mu_j(t) |^2
\end{equation}
where $\mu_j \in Y$ and $\xi=(t,z)\in X$.
We introduce a regularization term that penalizes the second derivative.
This is a common choice in the smoothing literature, e.g. \cite{randall99}.
The regularization term $r:Y^k\to [0,\infty)$ is given by
\begin{equation} \label{eq:5.2r}
r(\mu) = \sum_{j=1}^k \| \partial^2\mu_j\|_{L^2}^2.
\end{equation}
The $k$-means energy $f_n$ for data points $\{\xi_i=(t_i,z_i)\}_{i=1}^n$ is therefore written
\begin{equation} \label{eq:fnap}
f_n(\mu) = \frac{1}{n}\sum_{i=1}^n \bigwedge_{j=1}^k d(\xi_i,\mu_j) + \lambda r(\mu) = \frac{1}{n}\sum_{i=1}^n \bigwedge_{j=1}^k |z_i-\mu_j(t_i) |^2 + \lambda \sum_{j=1}^k \| \partial^2\mu_j\|_{L^2}^2.
\end{equation}

In most cases it is reasonable to assume that any minimizer of $f_\infty$ must be uniformly bounded, i.e. there exists $N$ (which will in general depend on $P$) such that if $\mu^\infty$ minimizes $f_\infty$ then $|\mu^\infty(t)|\leq N$ for all $t\in [0,T]$.
Under this assumption we redefine $Y$ to be
\begin{equation} \label{eq:Y}
Y=\{\mu_j\in H^2([0,T]) : |\mu_j(t)|\leq N \, \forall t\in [0,T]\}.
\end{equation}
Since pointwise evaluation is a bounded linear functional in $H^s$ (for $s\geq 1$) this space is weakly closed.
We now minimize $f_n$ over $Y^k$.
Note that we are not immediately guaranteed that minimizers of $f_n$ over $(H^s)^k$ are contained in $Y^k$.
However when we apply Theorem~\ref{thm:kmeanscons} we can conclude that minimizers $\mu^n$ of $f_n$ over $Y_k$ are weakly compact in $(H^s)^k$ and any limit point is a minimizer of $f_\infty$ in $Y^k$.
And therefore any limit point is a minimizer of $f_\infty$ over $(H^s)^k$.

If no such $N$ exists then our results in Theorem~\ref{thm:kmeanscons} are still valid however the minimum of $f_\infty$ over $(H^s)^k$ is not necessarily equal to the minimum of $f_\infty$ over $Y^k$.

Our results show that the $\Gamma$-limit for $\mathbb{P}$-almost every $\omega$ is
\begin{equation} \label{eq:finftyap}
f_\infty(\mu) = \int_X \bigwedge_{j=1}^k d(x,\mu_j) P(\text{d} x) + \lambda r(\mu) = \int_X \bigwedge_{j=1}^k |z-\mu_j(t) |^2 P(\text{d} x) + \lambda \sum_{j=1}^k \| \partial^2\mu_j\|_{L^2}^2.
\end{equation}
We start with the key result for this section, that is the existence of a weakly converging subsequence of minimizers.
Our result relies upon the regularity of Sobolev functions.
For our result to be meaningful we require that the minimizer should at least be continuous.
In fact every $g\in H^2([0,T])$ is in $C^s([0,T])$ for any $s<\frac{3}{2}$.
The regularity in the space allows us to further deduce the existence of a strongly converging subsequence.

\begin{theorem}
\label{thm:strconv}
Let $X=[0,T] \times \mathbb{R}^\kappa$ and define $Y$ by \eqref{eq:Y}.
Define $d:X \times Y\to [0,\infty)$ by \eqref{eq:5.2d} and $r:Y^k\to[0,\infty)$ by \eqref{eq:5.2r}.
For independent samples $\{\xi_i\}_{i=1}^n$ from $P$ which has compact support
$\tilde{X} \subset X$ define $f_n,f_\infty : Y^k\to \mathbb{R}$ by \eqref{eq:fnap} and \eqref{eq:finftyap} respectively.

Then (1) any sequence of minimizers $\mu^n\in Y^k$ of $f_n$ is $\mathbb{P}$-almost surely weakly-precompact (in $H^2$) with any weak limit point of $\mu^n$ minimizes $f_\infty$ and
%$\mathbb{P}$-almost every $\omega$, there exists a weakly converging (in $H^2$) subsequence to some $\mu^\infty\in Y^k$ that minimizes $f_\infty$ and
(2) if $\mu^{n_m}\rightharpoonup \mu$ is a weakly converging (in $H^2$) subsequence of minimizers then the convergence is uniform (in $C^0$).
\end{theorem}

To prove the first part of Theorem~\ref{thm:strconv} we are required to check the boundedness and continuity assumptions on $d$ (Proposition~\ref{prop:B2check}) and show that $r$ is weakly lower semi-continuous and coercive (Proposition~\ref{prop:C1check}).
This statement is then a straightforward application of Theorem~\ref{thm:kmeanscons}.
Note that we will have shown the result of Theorem~\ref{thm:gamcondiffspace} holds: $f_\infty = \Gamma\text{-}\lim_n f_n^{(\omega)}$.

In what follows we check that properties hold for any $x \in \tilde{X}$, which should be understood as implying that they hold for $P$-almost any $x \in X$; this is sufficient for our purposes as the collection of sequences $\xi_1,\ldots$ for which one or more observations lies in the complement of $\tilde{X}$ is $\mathbb{P}$-null and the support of $P_n$ is $\mathbb{P}$-almost surely contained within $\tilde{X}$.

\begin{proposition}
\label{prop:B2check}
Let $\tilde{X} = [0,T]\times[-\tilde{N},\tilde{N}]^\kappa$ and define $Y$ by \eqref{eq:Y}.
Define $d:\tilde{X}\times Y\to [0,\infty)$ by \eqref{eq:5.2d}.
Then (i) for all $y\in Y$ we have $\sup_{x\in \tilde{X}} d(x,y)<\infty$ and (ii) for any $x\in X$ and $y\in Y$ and any sequences $x_m\to x$ and $y_n\rightharpoonup y$ as $m,n\to \infty$ then we have $\liminf_{n,m\to \infty} d(x_m,y_n) = d(x,y)$.
\end{proposition}

\begin{proof}
We start with (i).
Let $y\in Y$ and $x = (t,z)\in [0,T]\times [-\tilde{N},\tilde{N}]^\kappa$, then
\begin{align*}
d(x,y) & = |z-y(t)|^2 \\
 & \leq 2|z|^2 + 2|y(t)|^2 \\
 & \leq 2\tilde{N}^2 + 2\sup_{t\in [0,T]} |y(t)|^2.
\end{align*}
Since $y$ is continuous then $\sup_{t\in [0,T]} |y(t)|^2< \infty$ and moreover we can bound $d(x,y)$ independently of $x$ which shows (i).

For (ii) we let $(t_m,z_m)=x_m \to x=(t,z)$ in $\mathbb{R}^{\kappa+1}$ and $y_n\rightharpoonup y$.
Then
\begin{align}
d(x_m,y_n) & = \left| z_m - y_n(t_m) \right|^2 \notag \\
 & = |z_m|^2 - 2 z_m \cdot y_n(t_m) + |y_n(t_m)|^2. \label{eq:dCts}
\end{align}
Clearly $|z_m|^2\to |z|^2$ and we now show that $y_n(t_m)\to y(t)$ as $m,n\to \infty$.

We start by showing that the sequence $\|y_n\|_Y$ is bounded.
Each $y_n$ can be associated with $\Lambda_n\in Y^{**}$ by $\Lambda_n(\nu)=\nu(y_n)$ for $\nu\in Y^*$.
As $y_n$ is weakly convergent it is weakly bounded.
So,
\[ \sup_{n\in\mathbb{N}} |\Lambda_n(\nu)| = \sup_{n\in \mathbb{N}} |\nu(y_n)| \leq M_\nu \]
for some $M_\nu<\infty$.
By the uniform boundedness principle \cite{conway90}
\[ \sup_{n\in\mathbb{N}} \|\Lambda_n\|_{Y^{**}} < \infty. \]
And so,
\[ \sup_{n\in\mathbb{N}} \|y_n\|_Y = \sup_{n\in\mathbb{N}} \|\Lambda_n\|_{Y^{**}} < \infty. \]
Hence there exists $M>0$ such that $\| y_n\|_Y \leq M$.
Therefore
\begin{align*}
| y_n(r) -y_n(s) | & = \left| \int_s^r \partial y_n(t)  \; \text{d} t \right|
\leq \int_s^r \left| \partial y_n(t) \right| \; \text{d} t  = \int_0^T \mathbb{I}_{[s,r]}(t) \left| \partial y_n(t) \right| \; \text{d} t \\
 & \leq \| \mathbb{I}_{[s,r]} \|_{L^2} \left\| \partial y_n(t) \right\|_{L^2}
\leq M \sqrt{|r-s|}.
\end{align*}
Since $y_n$ is uniformly bounded and equi-continuous then by the Arzel\`a--Ascoli theorem there exists a uniformly converging subsequence, say $y_{n_m}\to \hat{y}$.
By uniqueness of the weak limit $\hat{y}=y$.
But this implies that
\[ y_n(t) \to y(t) \]
uniformly for $t\in [0,T]$.
Now as
\[ |y_n(t_m) - y(t) | \leq |y_n(t_m) - y(t_m) | + | y(t_m) - y(t)| \]
then $y_n(t_m)\to y(t)$ as $m,n\to \infty$.
Therefore the second and third terms of~\eqref{eq:dCts} satisfies
\begin{align*}
2 z_m \cdot y_m(t_m) & \to 2z\cdot y(t) \\
\left|y_n(t_m)\right|^2 & \to \left|y(t)\right|^2
\end{align*}
as $m,n\to \infty$.
Hence
\[ d(x_m,y_n) \to |z|^2 - 2z\cdot y(t) + \left|y(t)\right|^2 = |z-y(t)|^2 = d(x,y) \]
which completes the proof.
\end{proof}

\begin{proposition}
\label{prop:C1check}
Define $Y$ by \eqref{eq:Y} and $r:Y^k\to [0,\infty)$ by \eqref{eq:5.2r}.
Then $r$ is weakly lower semi-continuous and coercive.
\end{proposition}

\begin{proof}
We start by showing $r$ is weakly lower semi-continuous.
For any weakly converging sequence $\mu^n_1\rightharpoonup \mu_1$ in $H^2$ we have that $\partial^2 \mu^n_1\rightharpoonup \partial^2 \mu_1$ weakly in $L^2$.
Hence it follows that $r$ is weakly lower semi-continuous.

To show $r$ is coercive let $\hat{r}(\mu_1)=\|\partial^2\mu_1\|_{L^2}^2$ for $\mu_1\in Y$.
We will show $\hat{r}$ is coercive.
Let $\mu_1\in Y$ and note that since $\mu_1\in C^1$ the first derivative exists (strongly).
Clearly we have $\|\mu_1\|_{L^2}\leq N\sqrt{T}$ and using a Poincar\'e inequality
\[ \left\|\frac{\text{d} \mu_1}{\text{d} t} - \frac{1}{T} \int_0^T \frac{\text{d} \mu_1}{\text{d} t} \; \text{d} t \right\|_{L^2} \leq C \|\partial^2 \mu_1\|_{L^2} \]
for some $C$ independent of $\mu_1$.
Therefore
\[ \left\|\frac{\text{d} \mu_1}{\text{d} t} \right\|_{L^2} \leq C \|\partial^2 \mu_1\|_{L^2} + \left|\frac{1}{T} \int_0^T \frac{\text{d} \mu_1}{\text{d} t} \; \text{d} t \right| \leq C\|\partial^2 \mu_1\|_{L^2} + \frac{2N}{T}. \]
It follows that if $\|\mu_1\|_{H^2}\to \infty$ then $\|\partial^2 \mu_1\|_{L^2}\to \infty$, hence $\hat{r}$ is coercive.
\end{proof}

Finally, the existence of a strongly convergent subsequence in Theorem~\ref{thm:strconv} follows from the fact that $H^2$ is compactly embedded into $H^1$.
Hence the convergence is strong in $H^1$.
By Morrey's inequality $H^1$ is embedded into a H\"older space ($C^{0,\frac{1}{2}}$) which is a subset of uniformly continuous functions.
This implies the convergence is uniform in $C^0$.

\section{Examples \label{sec:examples}}

In this section we give two exemplar applications of the methodology. In
principle any cost function, $d$, and regularization, $r$, (that satisfy the conditions) could be used.
For illustrative purposes we choose $d$ and $r$ to make the minimization simple to implement.
In particular, in Example 1 our choices allow us to use smoothing splines.% which are available inbuilt into many programs, e.g. MATLAB and R.

\subsection{Example 1: A Smoothing-Data Association Problem \label{sec:example1}}

We use the $k$-means method to solve a smoothing-data association problem.
For each $j = 1,2,\dots,k$ we take functions $x^j:[0,T]\times\mathbb{R}$ for $j=1,2,\dots ,k$ as the ``true'' cluster centers, and for sample times $t_i^j$ for $i=1,2,\dots n_j$, uniformly distributed over $[0,T]$, we let
\begin{equation} \label{eq:examples:example1:model}
z_i^j = x^j(t_i^j) + \epsilon_i^j
\end{equation}
where $\epsilon_i^j$ are iid noise terms.

The observations take the form $\xi_i = (t_i,z_i)$ for $i=1,2,\dots,n=\sum_{j=1}^k n_j$ where we have relabeled the observations to remove the (unobserved) target reference.
We model the observations with density (with respect to the Lebesgue measure)
\[ p((t,z)) = \frac{1}{T} \mathbb{I}_{[0,T]}(t) \sum_{j=1}^k w_j p_\epsilon(z-x^j(t)) \]
on $\mathbb{R}\times\mathbb{R}$ where $p_\epsilon$ denotes the common density of the $\epsilon_i^j$ and $w_j$ denotes the probability that an observation is generated by trajectory $j$.
We let each cluster center be equally weighted: $w_j=\frac{1}{k}$.
The cluster centers were fixed and in particular did not vary between numerical experiments.

When the noise is bounded this is precisely the problem described in Section~\ref{sec:MTT:theory} with $\kappa=1$, hence the problem converges.
We use a truncated Gaussian noise term.

In the theoretical analysis of the algorithm we have considered only the minimization problem associated with the $k$-means algorithm; of course minimizing complex functionals of the form of $f_n$ is itself a challenging problem.
Practically, we adopt the usual $k$-means strategy \cite{lloyd82} of iteratively assigning data to the closest of a collection of $k$ centers and then re-estimating each center by finding the center which minimizes the average regularized cost of the observations currently associated with that center.
As the energy function is bounded below and monotonically decreasing over
iterations, this algorithm converges to a local (but not necessarily global) minimum.

More precisely, in the particular example considered here we employ the
following iterative procedure:
\begin{enumerate}[1.]
\item Initialize $\varphi^0: \{1,2,\dots,n\} \to \{1,2,\dots,k\}$ arbitrarily.
\item For a given data partition $\varphi^r: \{1,2,\dots, n\}\to \{1,2,\dots, k\}$ we independently find the cluster centers $\mu^r=(\mu_1^r,\mu_2^r,\dots,\mu_k^r)$ where each $\mu_j^r\in H^2([0,T])$ by
\[ \mu^r_j = \argmin_{\mu_j} \frac{1}{n} \sum_{i:\varphi^r(i)=j} |z_i-\mu_j(t_i)|^2 + \lambda \|\partial^2 \mu_j\|_{L^2}^2 \quad \text{for } j=1,2,\dots, k. \]
This is done using smoothing splines.
\item Data is repartitioned using the cluster centers $\mu^r$
\[ \varphi^{r+1}(i) = \argmin_{j=1,2,\dots, k} |z_i - \mu_j^r(t_i)|. \]
\item If $\varphi^{r+1} \neq \varphi^{r}$ then return to Step 2.
Else we terminate.
\end{enumerate}

Let $\mu^n=(\mu^n_1,\dots,\mu^n_k)$ be the output of the $k$-means algorithm from $n$ data points.
To evaluate the success of the methodology when dealing with a finite sample of $n$ data points we look at how many iterations are required to reach convergence (defined as an assignment which is unchanged over the course of an algorithmic iteration), the number of data points correctly associated, the metric
\[ \eta(n) = \frac{1}{k} \sqrt{\sum_{j=1}^k \| \mu^n_j-x^j \|_{L^2}^2} \]
and the energy
\[ \hat{\theta}_n = f_n(\mu^n) \]
where
\[ f_n(\mu) = \frac{1}{n} \sum_{i=1}^n \bigwedge_{j=1}^k |z_i-\mu_j(t_i)|^2 + \lambda \sum_{j=1}^k \|\partial^2 \mu_j\|^2_{L^2}. \]

\begin{figure}%[htdp]
\caption{Smoothed data association trajectory results for the $k$-means method. \label{fig:ex4step}}
\centering
\setlength\figureheight{2.7cm}
\setlength\figurewidth{6cm}
\begin{tikzpicture}
\tikzstyle{longdash}=[dash pattern=on 6pt off 2pt]

\begin{axis}[%
width=\figurewidth,
height=\figureheight,
scale only axis,
xmin=0,
xmax=10,
ymin=-40,
ymax=60,
name=plot1,
axis x line*=bottom,
axis y line*=left
]
\addplot [color=black,mark size=1.0pt,only marks,mark=*,mark options={solid},forget plot]
  table[row sep=crcr]{example1kmeansv2-2.tsv};

\addplot [color=gray,solid,line width=1.0pt,forget plot]
  table[row sep=crcr]{example1kmeansv2-9.tsv};

\addplot [color=black,mark size=1.0pt,only marks,mark=*,mark options={solid},forget plot]
  table[row sep=crcr]{example1kmeansv2-5.tsv};

\addplot [color=gray,longdash,line width=1.0pt,forget plot]
  table[row sep=crcr]{example1kmeansv2-6.tsv};

\addplot [color=black,mark size=1.0pt,only marks,mark=*,mark options={solid},forget plot]
  table[row sep=crcr]{example1kmeansv2-8.tsv};

\addplot [color=gray,dotted,line width=1.0pt,forget plot]
  table[row sep=crcr]{example1kmeansv2-3.tsv};

\end{axis}

\begin{axis}[%
width=\figurewidth,
height=\figureheight,
scale only axis,
xmin=0,
xmax=10,
ymin=-40,
ymax=60,
name=plot2,
at=(plot1.right of south east),
anchor=left of south west,
axis x line*=bottom,
axis y line*=left
]
\addplot [color=black,solid,line width=2.0pt,forget plot]
  table[row sep=crcr]{example1kmeansv2-1.tsv};

\addplot [color=black,mark size=1.2pt,only marks,mark=triangle*,mark options={solid},forget plot]
  table[row sep=crcr]{example1kmeansv2-2.tsv};

\addplot [color=black,longdash,line width=2.0pt,forget plot]
  table[row sep=crcr]{example1kmeansv2-4.tsv};

\addplot [color=black,mark size=1.5pt,only marks,mark=star,mark options={solid},forget plot]
  table[row sep=crcr]{example1kmeansv2-5.tsv};

\addplot [color=black,dotted,line width=2.0pt,forget plot]
  table[row sep=crcr]{example1kmeansv2-7.tsv};

\addplot [color=black,mark size=1.2pt,only marks,mark=diamond*,mark options={solid},forget plot]
  table[row sep=crcr]{example1kmeansv2-8.tsv};

\end{axis}
\end{tikzpicture}
\caption*{
%\fontsize{20}{1}
The figure on the left shows the raw data with the data generating model.
That on the right shows the output of the $k$-means algorithm.
The parameters used are: $k=3$, $T=10$, $\epsilon_i^j$ from a $N(0,5)$ truncated at $\pm 100$, $\lambda=1$, $x^1(t)=-15-2t+0.2t^2$, $x^2(t)=5+t$ and $x^3(t)=40$.
}
\end{figure}

Figure~\ref{fig:ex4step} shows the raw data and output of the $k$-means algorithm for one realization of the model.
We run Monte Carlo trials for increasing numbers of data points; in particular we run $10^3$ numerical trials independently for each $n=300,600,\dots,3000$ where we generate the data from~\eqref{eq:examples:example1:model} and cluster using the above algorithm.
Each numerical experiment is independent.

\begin{figure}%[htdp]
\caption{Monte Carlo convergence results.}
\label{fig:resex4}
\centering
\setlength\figureheight{4cm}
\setlength\figurewidth{2.2cm}
\begin{subfigure}[FIGTOPCAP]{0.3\textwidth}
\caption{}
\begin{tikzpicture}
\begin{axis}[
width=\figurewidth,
height=\figureheight,
scale only axis,
xmin = 0, xmax = 3000,
ymin = 0, ymax = 35,
axis y line* = left, % the '*' avoids arrow heads
xlabel = {$n$},
xtick={0,1500,3000},
xlabel near ticks,
ylabel = {Iterations to Converge},
ylabel near ticks
]
% median iterations
\addplot [
color=black,
dotted,
line width=2.0pt,
forget plot
]
table[row sep=crcr]{
300 7\\
600 8\\
900 9\\
1200 10\\
1500 10\\
1800 10\\
2100 11\\
2400 11\\
2700 11\\
3000 10\\
};
% 25% quantile iterations
\addplot [
color=black,
dotted,
line width=1.0pt,
forget plot
]
table[row sep=crcr]{
300 4\\
600 5\\
900 5\\
1200 5\\
1500 5\\
1800 5\\
2100 5.5\\
2400 5\\
2700 6\\
3000 6\\
};
% 75% quantile iterations
\addplot [
color=black,
dotted,
line width=1.0pt,
forget plot
]
table[row sep=crcr]{
300 17\\
600 24.5\\
900 26\\
1200 26\\
1500 27\\
1800 27\\
2100 28\\
2400 27\\
2700 28\\
3000 27\\
};
\end{axis}
\begin{axis}[
width=\figurewidth,
height=\figureheight,
scale only axis,
xmin = 0, xmax = 3000,
ymin = 95, ymax = 100,
hide x axis,
axis y line*=right,
ylabel={\% Associated},
ylabel near ticks
]
% mean % Associated
\addplot [
color=black,
solid,
line width=2.0pt,
forget plot
]
table[row sep=crcr]{
300 99.67\\
600 99.67\\
900 99.56\\
1200 99.58\\
1500 99.60\\
1800 99.61\\
2100 99.62\\
2400 99.62\\
2700 99.59\\
3000 99.60\\
};
% 25% quantile % Associated
\addplot [
color=black,
solid,
line width=1.0pt,
forget plot
]
table[row sep=crcr]{
300 98.84\\
600 99.17\\
900 99.28\\
1200 99.33\\
1500 99.40\\
1800 99.39\\
2100 99.43\\
2400 99.42\\
2700 99.44\\
3000 99.47\\
};
% 75% quantile % Associated
\addplot [
color=black,
solid,
line width=1.0pt,
forget plot
]
table[row sep=crcr]{
300 100\\
600 99.83\\
900 99.89\\
1200 99.83\\
1500 99.80\\
1800 99.78\\
2100 99.76\\
2400 99.75\\
2700 99.74\\
3000 99.75\\
};
\end{axis}
\end{tikzpicture}
\end{subfigure}
\quad \quad
\begin{subfigure}[FIGTOPCAP]{0.3\textwidth}
\caption{}
\newcommand{\boxplot}[6]{%#max %no change
%#1: center, #2: median, #3: 1/4 quartile, #4: 3/4 quartile, #5: 5% quartile, #6: 95% quartile
\filldraw[fill=gray!25,line width=0.2mm] let \n{boxxl}={#1-90}, \n{boxxr}={#1+90} in (axis cs:\n{boxxl},#3) rectangle (axis cs:\n{boxxr},#4); % draw the box
%\draw[line width=0.2mm, color=black] let \n{boxxl}={#1-90}, \n{boxxr}={#1+90} in (axis cs:\n{boxxl},#2) -- (axis cs:\n{boxxr},#2); % median
\draw[line width=0.2mm,color=black] (axis cs:#1,#4) -- (axis cs:#1,#6); % bar up
\draw[line width=0.2mm,color=black] (axis cs:#1,#3) -- (axis cs:#1,#5); % bar down
}

\begin{tikzpicture}
\begin{axis}[
width=\figurewidth,
height=\figureheight,
scale only axis,
xmin = 0, xmax = 3100,
ymin = 0, ymax = 3,
%axis y line* = left, % the '*' avoids arrow heads
xlabel = {$n$},
xtick={0,1500,3000},
xlabel near ticks,
ylabel = {$\eta(n)$},
ylabel near ticks
]
\boxplot{300}{1.6859}{1.4664}{1.9760}{1.1951}{2.9}  %actual 95% quartile result is 13.4866
\boxplot{600}{1.3781}{1.2295}{1.5297}{1.0492}{2.0915}
\boxplot{900}{1.2495}{1.1360}{1.3724}{0.9903}{1.6129}
\boxplot{1200}{1.1848}{1.0871}{1.2865}{0.9569}{1.4495}
\boxplot{1500}{1.1508}{1.0634}{1.2384}{0.9567}{1.3700}
\boxplot{1800}{1.1162}{1.0474}{1.1970}{0.9494}{1.3366}
\boxplot{2100}{1.1036}{1.0344}{1.1744}{0.9399}{1.2936}
\boxplot{2400}{1.0884}{1.0263}{1.1517}{0.9405}{1.2660}
\boxplot{2700}{1.0701}{1.0133}{1.1332}{0.9174}{1.2363}
\boxplot{3000}{1.0599}{1.0030}{1.1198}{0.9179}{1.2117}
\draw[line width=0.2mm,color=black,->] (axis cs:300,2) -- (axis cs:300,2.9); % arrow to indicate line continues off the plot
%%% The table below is the median
\addplot [
color=black,
solid,
line width=2.0pt,
forget plot
]
table[row sep=crcr]{
300 1.6859\\
600 1.3781\\
900 1.2495\\
1200 1.1848\\
1500 1.1508\\
1800 1.1162\\
2100 1.1036\\
2400 1.0884\\
2700 1.0701\\
3000 1.0599\\
};
\end{axis}
\end{tikzpicture}
\end{subfigure}
\begin{subfigure}[FIGTOPCAP]{0.3\textwidth}
\caption{}
\begin{tikzpicture}
\tikzstyle{longdash}=[dash pattern=on 6pt off 2pt]
\begin{axis}[
width=\figurewidth,
height=\figureheight,
scale only axis,
xmin = 0, xmax = 3000,
ymin = 20, ymax = 30,
xlabel = {$n$},
xtick={0,1500,3000},
xlabel near ticks,
ytick={20,25,30},
ylabel = {$\hat{\theta}_n$},
ylabel near ticks
]
\addplot [
color=black,
solid,
line width=2.0pt,
forget plot
]
table[row sep=crcr]{
300 26.1888\\
600 26.0696\\
900 25.7552\\
1200 25.3901\\
1500 25.2111\\
1800 25.2725\\
2100 25.2330\\
2400 25.3677\\
2700 25.2252\\
3000 25.1907\\
};
\addplot [
color=black,
dashed,
line width=2.0pt,
forget plot
]
table[row sep=crcr]{
300 22.1191\\
600 23.1368\\
900 23.4829\\
1200 23.7124\\
1500 23.7841\\
1800 23.9699\\
2100 24.0711\\
2400 24.1275\\
2700 24.2123\\
3000 24.3128\\
};
\addplot [
color=black,
dashed,
line width=2.0pt,
forget plot
]
table[row sep=crcr]{
300 28.0417\\
600 26.9990\\
900 26.6013\\
1200 26.3974\\
1500 26.1784\\
1800 26.1109\\
2100 25.9797\\
2400 26.0249\\
2700 25.8833\\
3000 25.9300\\
};
\addplot [
color=black,
longdash,
line width=2.0pt,
forget plot
]
table[row sep=crcr]{
300 26.7441\\
600 26.6279\\
900 26.5731\\
1200 26.6381\\
1500 26.5868\\
1800 26.5938\\
2100 26.6025\\
2400 26.5830\\
2700 26.6068\\
3000 26.6359\\
};
\end{axis}
\end{tikzpicture}
\end{subfigure}
\caption*{
Convergence results for the parameters given in Figure~\ref{fig:ex4step}.
In (a) the thick dotted line corresponds to the median number of iterations taken for the method to converge and the thinner dotted lines are the 25\% and 75\% quantiles.
The thick solid line corresponds to the median percentage of data points correctly identified and the thinner solid line are the 25\% and 75\% quantiles.
(b) shows the median value of $\eta(n)$ (solid), interquartile range (box) and the interval between the 5\% and 95\% percentiles (whiskers).
(c) shows the mean minimum energy $\hat{\theta}_n$ (solid) and the 10\% and 90\% quantiles (dashed).
The energy associated with the data generating model is also shown (long dashes).
In order to increase the chance of finding a global minimum for each Monte Carlo trial ten different initializations were tried and the one that had the smallest energy on termination was recorded.
}
\end{figure}

Results, shown in Figure~\ref{fig:resex4}, illustrate that as measured by $\eta$ the performance of the $k$-means method improves with the size of the available data set, as do the proportion of data points correctly assigned.
The minimum energy stabilizes as the size of the data set increases, although the algorithm does take more iterations for the method to converge.
We also note that the energy of the data generating functions is higher than the minimum energy.

Since the iterative $k$-means algorithm described above does not necessarily
identify global minima, we tested the algorithm on
two targets whose paths intersect as shown in Figure \ref{fig:ex4switch}.
The data association hypotheses corresponding to correct and incorrect associations, after the crossing point, correspond to two local minima.
The observation window $[0,T]$ was expanded to investigate the convergence to the correct data association hypothesis.
To enable this to be described in more detail we introduce the crossing and non-crossing energies:
\begin{align*}
E_{\text{c}} & = \frac{1}{T} f_n(\mu_{\text{c}}) \\
E_{\text{nc}} & = \frac{1}{T} f_n(\mu_{\text{nc}})
\end{align*}
where $\mu_{\text{c}}$ and $\mu_{\text{nc}}$ are the $k$-means centers for the crossing (correct) and non-crossing (incorrect) solutions.
To allow the association performance to be quantified, we therefore define the relative energy
\[ \Delta E = E_{\text{c}} - E_{\text{nc}}. \]

\begin{figure}%[htdp]
\caption{Crossing tracks in the $k$-means method. \label{fig:ex4switch}}
\centering
\setlength\figureheight{1.8cm}
\setlength\figurewidth{5.8cm}
\input{example1correctswitch.tikz}
\input{example1incorrectswitch.tikz}
\caption*{
Typical data sets for times up to $T_{\text{max}}$ with cluster centers, fitted up till $T$, exhibiting crossing and non-crossing behavior.
The parameters used are $k=2$, $T_{\text{min}}=9.6\leq T\leq 11=T_{\text{max}}$, $\epsilon_i^j\iid N(0,5)$, $x^1(t)=-20+t^2$ and $x^2(t)=20+4t$.
There are $n=220$ data points uniformly distributed over $[0,11]$ with 110 observations for each track.
The crossing occurs at approximately $t\approx 8.6$ but we wait a further time unit before investigating the decision making procedure.
}
\end{figure}

\begin{figure}%[htdp]
\caption{Energy differences in the $k$-means method. \label{fig:ex4switchresults}}
\centering
\setlength\figureheight{3cm}
\setlength\figurewidth{5.8cm}
\begin{tikzpicture}
\begin{axis}[
width=\figurewidth,
height=\figureheight,
scale only axis,
xmin = 9.6, xmax = 11,
ymin = -6.5, ymax = 3.5,
axis y line* = left, % the '*' avoids arrow heads
xlabel = {$T$},
xlabel near ticks,
ylabel = {$\Delta E$},
ylabel near ticks
]
\addplot [
color=black,
solid,
line width=2.0pt,
forget plot
]
table[row sep=crcr]{
9.6 -4.2947\\
9.65 -3.4355\\
9.7 -2.9164\\
9.75 -2.3255\\
9.8 -2.0746\\
9.85 -1.8021\\
9.9 -1.4499\\
9.95 -1.3616\\
10.0 -1.2804\\
10.05 -1.2826\\
10.1 -1.2128\\
10.15 -1.0411\\
10.2 -1.0672\\
10.25 -0.8333\\
10.3 -0.8822\\
10.35 -0.9424\\
10.4 -0.7573\\
10.45 -0.6626\\
10.5 -0.6013\\
10.55 -0.5800\\
10.6 -0.4301\\
10.65 -0.4185\\
10.7 -0.2853\\
10.75 -0.1308\\
10.8 0.0957\\
10.85 0.1401\\
10.9 0.2905\\
10.95 0.4244\\
11.0 0.5248\\
};
% one sd below mean
\addplot [
color=black,
solid,
line width=1.0pt,
forget plot
]
table[row sep=crcr]{
9.6 -6.2602\\
9.65 -5.5058\\
9.7 -5.0778\\
9.75 -4.3372\\
9.8 -4.2722\\
9.85 -4.0405\\
9.9 -3.5139\\
9.95 -3.4826\\
10.0 -3.3896\\
10.05 -3.4727\\
10.1 -3.4763\\
10.15 -3.1006\\
10.2 -3.2543\\
10.25 -2.7685\\
10.3 -2.9848\\
10.35 -3.1537\\
10.4 -2.8353\\
10.45 -2.7180\\
10.5 -2.5833\\
10.55 -2.6178\\
10.6 -2.4713\\
10.65 -2.5806\\
10.7 -2.5149\\
10.75 -2.1304\\
10.8 -1.9762\\
10.85 -2.0261\\
10.9 -1.9711\\
10.95 -1.7948\\
11.0 -1.7120\\
};
% one sd above mean
\addplot [
color=black,
solid,
line width=1.0pt,
forget plot
]
table[row sep=crcr]{
9.6 -1.3292\\
9.65 -1.3653\\
9.7 -0.7550\\
9.75 -0.3138\\
9.8 0.1230\\
9.85 0.4363\\
9.9 0.6140\\
9.95 0.7594\\
10.0 0.8289\\
10.05 0.9076\\
10.1 1.0508\\
10.15 1.0184\\
10.2 1.1200\\
10.25 1.1020\\
10.3 1.2204\\
10.35 1.2690\\
10.4 1.3207\\
10.45 1.3928\\
10.5 1.3807\\
10.55 1.4578\\
10.6 1.6111\\
10.65 1.7436\\
10.7 1.9443\\
10.75 1.8687\\
10.8 2.1676\\
10.85 2.3062\\
10.9 2.5521\\
10.95 2.6435\\
11.0 2.7617\\
};
\end{axis}
\begin{axis}[
width=\figurewidth,
height=\figureheight,
scale only axis,
xmin = 9.7, xmax = 11,
ymin = 0, ymax = 100,
hide x axis,
axis y line*=right,
ylabel={\% Correctly identified},
ylabel near ticks
]
\addplot [
color=black,
dashed,
line width=2.0pt,
forget plot
]
table[row sep=crcr]{
9.6 12.7\\
9.65 16.5\\
9.7 20.6\\
9.75 23.3\\
9.8 29.4\\
9.85 36.9\\
9.9 40.4\\
9.95 45.2\\
10.0 48.4\\
10.05 53.4\\
10.1 57.2\\
10.15 59.3\\
10.2 61.4\\
10.25 64.0\\
10.3 64.0\\
10.35 66.0\\
10.4 64.5\\
10.45 66.0\\
10.5 65.7\\
10.55 67.0\\
10.6 65.0\\
10.65 65.1\\
10.7 65.5\\
10.75 65.2\\
10.8 64.0\\
10.85 63.4\\
10.9 62.2\\
10.95 61.1\\
11.0 61.4\\
};
\end{axis}
\end{tikzpicture}
\caption*{
Mean results are shown for data obtained using the parameters given in Figure~\ref{fig:ex4switch} for data up to time $T$ (between $T_{\text{min}}$ and $T_{\text{max}}$).
The thick solid line shows the mean $\Delta E$ and the thinner lines one standard deviation either side of the mean.
The dashed line shows the percentage of times we correctly identified the tracks as crossing.
}
\end{figure}

To determine how many numerical trials we should run in order to get a good number of simulations that produce crossing and non-crossing outputs we first ran the experiment until we achieved at least 100 tracks that crossed and at least 100 that did not.
I.e. let $N_t^\text{c}$ be the number of trials that output tracks that crossed and $N_t^\text{nc}$ be the number of trials that output tracks that did not cross.
We stop when $\text{min}\{N_t^\text{c},N_t^\text{nc}\}\geq 100$.
Let $N_t = 10\left(N_t^\text{c} + N_t^\text{c}\right)$.
We then re-ran the experiment with $N_t$ trials so we expect that we get 1000 tracks that do not cross and 1000 tracks that do cross at each time $t$.
%I.e. $\min\{\mathbb{E}(N_t^\text{c}),\mathbb{E}(N_t^\text{nc})\}=1000$.

The results in Figure~\ref{fig:ex4switchresults} show that initially the better solution to the $k$-means minimization problem is the one that incorrectly partitions the tracks after the intersection.
However, as time is run forward the $k$-means favors the partition that correctly associates tracks to targets.
This is reflected in both an increase in $\Delta E$ and the percentage of outputs that correctly identify the switch.
Our results show that for $T>9.7$ the energy difference between the two minima grows linearly with time.
However, when we look which minima the $k$-means algorithm finds our results suggest that after time $T\approx 10.25$ the probability of finding the correct minima stabilizes at approximately 64\%.
There is reasonably large variance in the energy difference.
The mean plus standard deviation is positive for all $T$ greater than 9.8, however it takes until $T=10.8$ for the average energy difference to be positive.

\subsection{Example 2: Passive Electromagnetic Source Tracking}

In the previous example the data is simply a linear projection of the trajectories.
In contrast, here we consider the more general case where the measurement $X$ and model $Y$ spaces are very different; being connected by a complicated mapping that results in a very non-linear cost function $d$.
While the increased complexity of the cost function does lead to a (linear in data size) increase in computational cost, the problem is equally amenable to our approach.

In this example we consider the tracking of targets that periodically emit radio pulses as they travel on a two dimensional surface.
These emissions are detected by an array of (three) sensors that characterize the detected emissions in terms of `time of arrival', `signal amplitude' and the `identity of the sensor making the detection'.

Expressed in this way, the problem has a structure which does not fall directly within the framework which the theoretical results of previous sections cover.
In particular, the observations are not independent (we have exactly one from each target in each measurement interval), they are not identically distributed and they do not admit an empirical measure which is weakly convergent in the large data limit.

This formulation could be refined so that the problem did fall precisely within the framework; but only at the expense of losing physical clarity.
This is not done but as shall be seen below, even in the current formulation, good performance is obtained.
This gives some confidence that $k$-means like strategies in general settings, at least when the qualitatively important features of the problem are close to those
considered theoretically, and gives some heuristic justification for the lack of rigor.

Three sensors receive amplitude and time of arrival from each target with periodicity $\tau$.
Data at each sensor are points in $\mathbb{R}^2$ whilst the cluster centers (trajectories) are time-parameterized curves in a different $\mathbb{R}^2$ space.

In the generating model, for clarity we again index the targets in the observed amplitude and time of arrival.
However, we again assume that this identifier is not observed and this notation is redefined (identities suppressed) when we apply the $k$-means method.

Let $x_j(t)\in \mathbb{R}^2$ be the position of target $j$ for $j=1,2,\dots k$ at time $t\in [0,T]$.
In every time frame of length $\tau$ each target emits a signal which is detected at three sensors.
The time difference from the start of the time frame to when the target emits this signal is called the time offset.
The time offset for each target is a constant which we call $o_j$ for $j=1,2,\dots, k$.
Target $j$ therefore emits a signal at times
\[ \tilde{t}_j(m) = m\tau + o_j \]
for $m\in \mathbb{N}$ such that $\tilde{t}_j(m)\leq T$.
Note that this is not the time of arrival and we do not observe $\tilde{t}_j(m)$.

Sensor $p$ at position $z_p$ detects this signal some time later and measures the time of arrival $t_j^p(m)\in [0,T]$ and amplitude $a^p_j(m)\in \mathbb{R}$ from target $j$.
The time of arrival is
\[ t^p_j(m) = m \tau + o_j + \frac{| x_j(m) - z_p |}{c} + \epsilon^p_j(m) = \tilde{t}_j(m) + \frac{| x_j(m) - z_p |}{c} + \epsilon^p_j(m) \]
where $c$ is the speed of the signal and $\epsilon^p_j(m)$ are iid noise terms with variance $\sigma^2$.
The amplitude is
\[ a^p_j(m) = \log\left( \frac{\alpha}{| x_j(m) - z_p |^2 + \beta} \right) + \delta^p_j(m) \]
where $\alpha$ and $\beta$ are constants and $\delta^p_j(m)$ are iid noise terms with variance $\nu^2$.
We assume the parameters $\alpha$, $\beta$, $c$, $\sigma$, $\tau$, $\nu$ and $z_p$ are known.

To simplify the notation $\Pi_q x:\mathbb{R}^2\to \mathbb{R}$ is the projection of $x$ onto it's $q^\text{th}$ coordinate for $q=1,2$.
I.e. the position of target $j$ at time $t$ can be written $x_j(t)=(\Pi_1 x_j(t), \Pi_2 x_j(t))$.

In practice we do not know to which target each observation corresponds.
We use the $k$-means method to partition a set $\{\xi_i=(t_i,a_i,p_i)\}_{i=1}^{n}$ into the $k$ targets.
Note the relabeling of indices; $\xi_i=(t_i,a_i,p_i)$ is the time of arrival $t_i$, amplitude $a_i$ and sensor $p_i$ of the $i^\text{th}$ detection.
The cluster centers are in a function-parameter product space $\mu_j = (\hat{x}_j(t),\hat{o}_j)\in C^0([0,T];\mathbb{R}^2)\times [0,\tau) \subset C^0([0,T];\mathbb{R}^2) \times \mathbb{R}$ that estimates the $j^\text{th}$ target's trajectory and time offset.
The $k$-means minimization problem is
\[ \mu^n = \argmin_{\mu\in (C^0\times [0,\tau))^k} \frac{1}{n} \sum_{i=1}^n \bigwedge_{j=1}^k d(\xi_i,\mu_j) \]
for a choice of cost function $d$.
If we look for cluster centers as straight trajectories then we can restrict ourselves to functions of the form $x_j(t) = x_j(0) + v_j t$ and consider the cluster centers as finite dimensional objects.
This allows us to redefine our minimization problem as
\[ \mu^n = \argmin_{\mu\in (\mathbb{R}^4\times[0,\tau))^k} \frac{1}{n} \sum_{i=1}^j \bigwedge_{j=1}^k d(\xi_i,\mu_j) \]
so that now $\mu_j = (x_j(0),v_j,o_j)\in \mathbb{R}^2\times\mathbb{R}^2\times[0,\tau)$.
We note that in this finite dimensional formulation it is not necessary to include a regularization term; a feature already anticipated in the definition of the minimization problem.

For $\mu_j=(x_j,v_j,o_j)$ we define the cost function
\[ d((t,a,p),\mu_j) = \left( \left( t,a\right) - \psi(\mu_j,p,m) \right) \left( \begin{array}{cc} \frac{1}{\sigma^2} & 0 \\ 0 & \frac{1}{\nu^2} \end{array} \right) \left( \left( \begin{array}{c} t \\ a \end{array} \right) - \psi(\mu_j,p,m)^\top  \right) \]
where $m=\max\{n\in \mathbb{N}: n\tau \leq t\}$,
\[ \psi(\mu_j,p,m) = \left(\frac{|x_j+m \tau v_j - z_p|}{c} + o_j + m\tau, \log\left( \frac{\alpha}{|x_j+m\tau v_j-z_p|^2 + \beta} \right) \right) \]
and superscript $T$ denotes the transpose.

We initialize the partitions by uniformly randomly choosing $\varphi^0:\{1,2,\dots, n\}\to \{1,2,\dots ,k\}$.
At the $r^{\text{th}}$ iteration the $k$-means minimization problem is then partitioned into $k$ independent problems
\[ \mu_j^r = \argmin_{\mu_j} \sum_{i\in (\varphi^{r-1})^{-1}(j)} d((t_i,a_i,p_i),\mu_j^0) \quad \text{for } 1\leq j\leq k. \]
A range of initializations for $\mu_j$ are used to increase the chance of the method converging to a global minimum.

For optimal centers conditioned on partition $\varphi^{r-1}$ we can define the partition $\varphi^r$ to be the optimal partition of $\{(t_i,a_i,p_i)\}_{i=1}^{n}$ conditioned on centers $(\mu_j^r)$ by solving
\begin{align*}
\varphi^r: \{1,2,\dots, n\} & \to \{1,2,\dots, k\} \\
i & \mapsto \argmin_{j=1,2,\dots,k} d((t_i,a_i,p_i),\mu_j^r).
\end{align*}
The method has converged when $\varphi^r=\varphi^{r-1}$ for some $r$.
Typical simulated data and resulting trajectories are shown in Figure~\ref{fig:res}.

\begin{figure}[htdp!]
\caption{Representative data and resulting tracks for the passive tracking example. \label{fig:res}}
\centering
\setlength\figureheight{3cm}
\setlength\figurewidth{5.8cm}
\input{example2kmeans.tikz}
\caption*{
Representative data is shown for the parameters $k=2$, $\tau=1$, $T=1000$, $c=100$, $z_1=(-10,-10)$, $z_2=(10,-10)$, $z_3=(0,10)$, $\epsilon^p_j(m)\iid N(0,0.03^2)$, $\delta^p_j(m) \iid N(0,0.05^2)$, $\alpha=10^8$, $\beta=5$, $x_1(t)=\frac{\sqrt{2}t}{400}(1,1)+(0,5)$, $x_2(t)=(6,7)-\frac{t}{125}(1,0)$, $o_1=0.3$ and $o_2=0.6$, given the sensor configuration shown at the top of the figure.
The $k$-means method was run until it converged, with the trajectory component of the resulting cluster centers plotted with the true trajectories at the top of the figure.
Target one is the dashed line with starred data points, target two is the solid line and square data points.
}
\end{figure}

To illustrate the convergence result achieved above we performed a test on a set of data simulated from the same model as in Figure~\ref{fig:res}.
We sample $n_s$ observations from $\{(t_i,a_i,p_i)\}_{i=1}^{n}$ and compare our results as $n_s \to n$.
Let $\hat{x}^{n_s}(t)=(\hat{x}_1^{n_s}(t),\dots,\hat{x}_k^{n_s}(t))$ be the position output by the $k$-means method described above using $n_s$ data points and $x(t)=(x_1(t),\dots, x_k(t))$ be the true values of each cluster center.
We use the metric
\[ \eta(n_s) = \frac{1}{k} \sqrt{\sum_{j=1}^k\|\hat{x}_j^{n_s}-x_j\|_{L^2}^2} \]
to measure how close the estimated position is to the exact position.
Note we do not use the estimated time offset given by the first model.
The number of iterations required for the method to converge is also recorded.
Results are shown in Figure~\ref{fig:resex2}.

In this example the data has enough separation that we are always able to recover the true data partition.
We also see improvement in our estimated cluster centers and convergence of the minimum energy as we increase the size of the data.
Finding global minima is difficult and although we run the $k$-means method from multiple starting points we sometimes only find local minima.
For $\frac{n_s}{n}=0.3$ we see the effect of finding local minima.
In this case only one Monte Carlo trial produces a bad result, but the error $\eta$ is so great (around 28 times greater than the average) that it can be seen in the mean result shown in Figure~\ref{fig:resex2}(c).

\begin{figure}[htdp]
\caption{Monte Carlo convergence results.}
\label{fig:resex2}
\centering
\setlength\figureheight{4cm}
\setlength\figurewidth{2.2cm}
\begin{subfigure}[FIGTOPCAP]{0.3\textwidth}
\caption{}
\begin{tikzpicture}
\tikzstyle{longdash}=[dash pattern=on 6pt off 2pt]
\begin{axis}[
width=\figurewidth,
height=\figureheight,
scale only axis,
xmin = 0, xmax = 1,
ymin = 1, ymax = 5,
axis y line* = left, % the '*' avoids arrow heads
xlabel = {$n_s / n$},
xtick={0,0.5,1},
xlabel near ticks,
ylabel = {Iterations to Converge},
ylabel near ticks
]
% mean Iterations
\addplot [
color=black,
longdash,
line width=2.0pt,
forget plot
]
table[row sep=crcr]{
0.1 2.5220\\
0.2 2.5480\\
0.3 2.5780\\
0.4 2.5260\\
0.5 2.6460\\
0.6 2.5690\\
0.7 2.5400\\
0.8 2.5610\\
0.9 2.5900\\
1 2.5670\\
};
% median Iterations
\addplot [
color=black,
dotted,
line width=2.0pt,
forget plot
]
table[row sep=crcr]{
0.1 2\\
0.2 2\\
0.3 2\\
0.4 2\\
0.5 2\\
0.6 2\\
0.7 2\\
0.8 2\\
0.9 2\\
1 2\\
};
% 25% quantile Iterations
\addplot [
color=black,
dotted,
line width=1.0pt,
forget plot
]
table[row sep=crcr]{
0.1 2\\
0.2 2\\
0.3 2\\
0.4 2\\
0.5 2\\
0.6 2\\
0.7 2\\
0.8 2\\
0.9 2\\
1 2\\
};
% 75% quantile Iterations
\addplot [
color=black,
dotted,
line width=1.0pt,
forget plot
]
table[row sep=crcr]{
0.1 4\\
0.2 4\\
0.3 4\\
0.4 4\\
0.5 4\\
0.6 4\\
0.7 4\\
0.8 4\\
0.9 4\\
1 4\\
};
\end{axis}
\begin{axis}[
width=\figurewidth,
height=\figureheight,
scale only axis,
xmin = 0, xmax = 1,
ymin = 95, ymax = 101,
hide x axis,
axis y line*=right,
ylabel={\% Associated},
ylabel near ticks
]
% median % Associated 
\addplot [
color=black,
solid,
line width=2.0pt,
forget plot
]
table[row sep=crcr]{
0.1 100\\
0.2 100\\
0.3 100\\
0.4 100\\
0.5 100\\
0.6 100\\
0.7 100\\
0.8 100\\
0.9 100\\
1 100\\
};
% 25% quantile % Associated
\addplot [
color=black,
dotted,
line width=1.0pt,
forget plot
]
table[row sep=crcr]{
0.1 100\\
0.2 100\\
0.3 100\\
0.4 100\\
0.5 100\\
0.6 100\\
0.7 100\\
0.8 100\\
0.9 100\\
1 100\\
};
% 75% quantile % Associated
\addplot [
color=black,
dotted,
line width=1.0pt,
forget plot
]
table[row sep=crcr]{
0.1 100\\
0.2 100\\
0.3 100\\
0.4 100\\
0.5 100\\
0.6 100\\
0.7 100\\
0.8 100\\
0.9 100\\
1 100\\
};
\end{axis}
\end{tikzpicture}
\end{subfigure}
\quad \quad
\begin{subfigure}[FIGTOPCAP]{0.3\textwidth}
\caption{}
\newcommand{\boxplot}[6]{%#max %no change
%#1: center, #2: median, #3: 1/4 quartile, #4: 3/4 quartile, #5: min, #6: max
\filldraw[fill=gray!25,line width=0.2mm] let \n{boxxl}={#1-0.03}, \n{boxxr}={#1+0.03} in (axis cs:\n{boxxl},#3) rectangle (axis cs:\n{boxxr},#4); % draw the box
%\draw[line width=0.2mm, color=black] let \n{boxxl}={#1-0.03}, \n{boxxr}={#1+0.03} in (axis cs:\n{boxxl},#2) -- (axis cs:\n{boxxr},#2); % median
\draw[line width=0.2mm,color=black] (axis cs:#1,#4) -- (axis cs:#1,#6); % bar up
\draw[line width=0.2mm,color=black] (axis cs:#1,#3) -- (axis cs:#1,#5); % bar down
}

\begin{tikzpicture}
\begin{axis}[
width=\figurewidth,
height=\figureheight,
scale only axis,
xmin = 0, xmax = 1.1,
ymin = 0, ymax = 2.5,
%axis y line* = left, % the '*' avoids arrow heads
xlabel = {${n_s}/{n}$},
xtick={0,0.5,1},
xlabel near ticks,
ylabel = {$\eta(n)$},
ylabel near ticks
]
\boxplot{0.1}{1.3171}{1.0133}{1.6915}{0.6909}{2.2822}
\boxplot{0.2}{0.9286}{0.7129}{1.1840}{0.4814}{1.6597}
\boxplot{0.3}{0.7661}{0.5896}{0.9871}{0.3956}{1.4156}
\boxplot{0.4}{0.6494}{0.5046}{0.8362}{0.3446}{1.1904}
\boxplot{0.5}{0.5779}{0.4528}{0.7304}{0.3000}{0.9833}
\boxplot{0.6}{0.5530}{0.4287}{0.7144}{0.2988}{0.9755}
\boxplot{0.7}{0.5088}{0.3992}{0.6422}{0.2649}{0.8795}
\boxplot{0.8}{0.4805}{0.3724}{0.6157}{0.2573}{0.8421}
\boxplot{0.9}{0.4647}{0.3570}{0.5992}{0.2433}{0.8128}
\boxplot{1}{0.4263}{0.3333}{0.5359}{0.2296}{0.7481}
%%% The table below is the mean
%\addplot [
%color=black,
%solid,
%line width=2.0pt,
%forget plot
%]
%table[row sep=crcr]{
%0.1 1.4976\\
%0.2 1.0346\\
%0.3 1.0171\\
%0.4 0.7293\\
%0.5 0.6402\\
%0.6 0.6190\\
%0.7 0.5510\\
%0.8 0.5426\\
%0.9 0.4965\\
%1 0.4604\\
%};
%%% The table below is the median
\addplot [
color=black,
solid,
line width=2.0pt,
forget plot
]
table[row sep=crcr]{
0.1 1.3171\\
0.2 0.9286\\
0.3 0.7661\\
0.4 0.6494\\
0.5 0.5779\\
0.6 0.5530\\
0.7 0.5088\\
0.8 0.4805\\
0.9 0.4647\\
1 0.4263\\
};
\end{axis}
\end{tikzpicture}
\end{subfigure}
\begin{subfigure}[FIGTOPCAP]{0.3\textwidth}
\caption{}
\begin{tikzpicture}
\tikzstyle{longdash}=[dash pattern=on 6pt off 2pt]
\begin{axis}[
width=\figurewidth,
height=\figureheight,
scale only axis,
xmin = 0, xmax = 1,
ymin = 1.8, ymax = 2.2,
xlabel = {$n_s / n$},
xtick={0,0.5,1},
xlabel near ticks,
ylabel = {$\hat{\theta}_n$},
ylabel near ticks
]
\addplot [
color=black,
solid,
line width=2.0pt,
forget plot
]
table[row sep=crcr]{
0.1 1.9916\\
0.2 1.9965\\
0.3 2.0176\\
0.4 1.9967\\
0.5 2.0001\\
0.6 1.9989\\
0.7 1.9977\\
0.8 2.0004\\
0.9 1.9986\\
1 1.9993\\
};
\addplot [
color=black,
dashed,
line width=2.0pt,
forget plot
]
table[row sep=crcr]{
0.1 1.8788\\
0.2 1.9189\\
0.3 1.9367\\
0.4 1.9434\\
0.5 1.9529\\
0.6 1.9516\\
0.7 1.9588\\
0.8 1.9602\\
0.9 1.9638\\
1 1.9652\\
};
\addplot [
color=black,
dashed,
line width=2.0pt,
forget plot
]
table[row sep=crcr]{
0.1 2.0987\\
0.2 2.0657\\
0.3 2.0595\\
0.4 2.0458\\
0.5 2.0425\\
0.6 2.0392\\
0.7 2.0358\\
0.8 2.0379\\
0.9 2.0344\\
1 2.0336\\
};
\addplot [
color=black,
longdash,
line width=2.0pt,
forget plot
]
table[row sep=crcr]{
0.1 1.9997\\
0.2 1.9983\\
0.3 2.0022\\
0.4 1.9986\\
0.5 2.0008\\
0.6 2.0000\\
0.7 1.9990\\
0.8 2.0010\\
0.9 2.0006\\
1 2.0008\\
};
\end{axis}
\end{tikzpicture}
\end{subfigure}
\caption*{
Convergence results for $10^3$ Monte Carlo trials with the parameters given in Figure~\ref{fig:res}; expressed with the notation used in Figure~\ref{fig:resex4}.
In (a) we have also recorded the mean number of iterations to converge (long dashes).
The 25\% and 75\% quantiles for the number of iterations to converge is 2 and 4 for all $n$ respectively.
The 25\% and 75\% quantiles for the percentage of data points correctly identified is 100\% in both cases for all $n$.
This is due to large separation in the data space.
To increase the chance of finding a global minimum for each Monte Carlo trial, out of five different initializations, that which had the smallest energy on terminating was recorded.
}
\end{figure}

\section*{Acknowledgments}

The authors are grateful for two anonymous reviewers' valuable comments which significantly improved the manuscript.
MT is part of MASDOC at the University of Warwick and was supported by an EPSRC Industrial CASE Award PhD Studentship with Selex-ES Ltd.

%\appendix
%\appendixpage
%
%\section{}

\nocite{cuesta88,wu83,biau08,canas12}
\bibliographystyle{plain}
\bibliography{references}

\end{document}